\documentclass[oneside,english]{amsart}
\usepackage[T1]{fontenc} 
\usepackage[utf8]{inputenc}
\usepackage{xcolor}
\usepackage{amstext}
\usepackage{amsthm}
\usepackage{amssymb}
\usepackage[all]{xy}
\usepackage{enumitem}
\usepackage{youngtab}

\usepackage{float}
\restylefloat{table}

\usepackage[unicode=true,bookmarks=true,bookmarksnumbered=false,bookmarksopen=true,bookmarksopenlevel=1,breaklinks=false,backref=false,colorlinks=true,linkcolor=blue,citecolor=cyan]{hyperref}

\usepackage{ifpdf}
\IfFileExists{lmodern.sty}{\usepackage{lmodern}}{}

\makeatletter

\hypersetup{pdftitle={\@title},pdfauthor={\@author}}

\numberwithin{equation}{section}
\theoremstyle{plain}
\newtheorem{thm}{\protect\theoremname}[section]
  \theoremstyle{plain}
  \newtheorem{cor}[thm]{\protect\corollaryname}
  \theoremstyle{plain}
  \newtheorem{lem}[thm]{\protect\lemmaname}
  \theoremstyle{plain}
  \newtheorem{prop}[thm]{\protect\propositionname}
  \theoremstyle{remark}
  \newtheorem{rem}[thm]{\protect\remarkname}
  \theoremstyle{remark}
  \newtheorem{claim}[thm]{\protect\claimname}
  \theoremstyle{remark}
  \newtheorem{exa}[thm]{\protect\examplename}

\AtBeginDocument{
  
}

\makeatother

\usepackage{babel}
  \providecommand{\claimname}{Claim}
  \providecommand{\corollaryname}{Corollary}
  \providecommand{\lemmaname}{Lemma}
  \providecommand{\propositionname}{Proposition}
  \providecommand{\remarkname}{Remark}
  \providecommand{\examplename}{Example}
\providecommand{\theoremname}{Theorem}

\begin{document}

\title{Symmetric and Exterior Squares of Hook Representations}

\author{Szabolcs Mészáros, János Wolosz}
\begin{abstract}
We determine the multiplicities of irreducible summands in the symmetric and the exterior squares of hook representations of symmetric groups over a field of characteristic zero. 
\end{abstract}



\thanks{This research was partially supported by National Research, Development and Innovation Office, NKFIH grants no. K115799 and K119934.}

\address{Eötvös Loránd University, 1053 Budapest, Egyetem tér 1-3, Hungary}

\email{janos.wolosz@gmail.com}

\address{MTA Rényi Institute, 1053 Budapest, Reáltanoda utca 13-15, Hungary}

\email{meszaros.szabolcs@renyi.hu}
\maketitle

\section{Introduction}

Let $\mathbb{F}$ be a field of characteristic zero and $n$ a positive integer. 
Denote by $M^{\lambda}$ the irreducible right $\mathbb{F}$-representation of the symmetric group $\mathfrak{S}_{n}$ corresponding to the Young diagram $\lambda\vdash n$ 
(i.e. $\lambda$ is a finite, non-increasing sequence of positive integers that add up to $n$). The multiplicities of irreducibles in $M^{\lambda}\otimes M^{\mu}$ are called Kronecker coefficients, their study is an active area of research, see \cite{Blasiak},\cite{Rosas2},\cite{Hayashi},\cite{SamSnowden}.

Let $V=M^{(n-1,1)}$ be the standard $\mathfrak{S}_n$-representation of degree $n-1$ over $\mathbb{F}$. In \cite[Thm. 1.2]{Remmel1} J. B. Remmel determined the multiplicities of irreducible summands of $\Lambda^k V \otimes \Lambda^k V$, for all $n,k\in\mathbb{N}^{+}$. (The statement is spelled out in detail in Subsection \ref{subsec:Remmel's-theorem}).
The representations of the form $\Lambda^k V$ are called hook representations, since $\Lambda^{k}V\cong M^{(n-k,1^{k})}$, so the Young diagram of $\Lambda^k V$ resembles a hook.
The factors appearing in the decomposition are either hook representations themselves i.e. $\lambda_{2}\leq 1$ or "double hook" representations i.e. $\lambda_{3}\leq2$ (but $\lambda_{2}>1$).
In this paper, we refine this decomposition by separating the summands of the symmetric and antisymmetric components. In other words, we determine the multiplicites of irreducible summands in the representations $\mathrm{Sym}^{2}(\Lambda^{k}V)$ and $\Lambda^{2}(\Lambda^{k}V)$.

Consider $\mathrm{Sym}^{2}(\Lambda^{k}V)$ and $\Lambda^{2}(\Lambda^{k}V)$
as complementary subspaces of $(\Lambda^{k}V)^{\otimes2}$, and denote by $(q,p,2^{d_2},1^{d_1})$ the Young diagram 
$(q,p,\underbrace{2,\dots,2}_{d_2},\underbrace{1,\dots,1}_{d_1})\vdash n$.

\vspace{-12pt}We show the following:
\begin{thm}
\label{thm:main}
Let $M^{\lambda}$ be an irreducible summand of $(\Lambda^{k}V)^{\otimes2}$ for some $\lambda\vdash n$.
\begin{itemize}
\item \setlength\itemsep{4pt}If $\lambda=(q,p,2^{d_{2}},1^{d_{1}})$ where
$q\geq p\geq2$ and
\begin{itemize}
\item \setlength\itemsep{4pt}$d_{1}\overset{\vphantom{k}}{\equiv}0\ mod \ 4$,
then every $M^{\lambda}$ factor is contained in $\mathrm{Sym}^{2}(\Lambda^{k}V)$,
\item $d_{1}\equiv2\ mod \ 4$, 
then every $M^{\lambda}$ factor is contained
in $\Lambda^{2}(\Lambda^{k}V)$,
\item $2\nmid d_{1}$, 
then the multiplicity of $M^{\lambda}$ is $1$
in $\mathrm{Sym}^{2}(\Lambda^{k}V)$ and in $\Lambda^{2}(\Lambda^{k}V)$,
\end{itemize}
\item If $\lambda=(n-m,1^{m})$ where $0\leq m\leq n-1$ and 
\begin{itemize}
\item \setlength\itemsep{4pt}$m\overset{\vphantom{k}}{\equiv}0\textrm{ or }1\ mod \ 4$,
then every $M^{\lambda}$ factor is contained in $\mathrm{Sym}^{2}(\Lambda^{k}V)$,
\item $m\equiv2\textrm{ or }3\ mod \ 4$, 
then every $M^{\lambda}$ factor is
contained in $\Lambda^{2}(\Lambda^{k}V)$.
\end{itemize}
\end{itemize}
\end{thm}

The theorem is expounded with explicit coefficients in Corollary \ref{cor:detailed}.

Note that if $\lambda$ is a double hook, then the multiplicity of $M^{\lambda}$ in $\mathrm{Sym}^{2}(\Lambda^{k}V)$ or $\Lambda^{2}(\Lambda^{k}V)$ depends only on the modulo $4$ value of $d_{1}$, the \emph{tail} of the Young diagram, independently of $q$, $p$, $d_2$, if the multiplicity of $M^{\lambda}$ in $(\Lambda^{k}V)^{\otimes2}$ is given. 
The phenomenon is comparable to Murnaghan's theorem (see \cite{Murnaghan1}, \cite{Murnaghan2}) 
if we rephrase it as follows: for $\lambda=(n-p-2d_2-d_1,p,2^{d_2},1^{d_1})$ the multiplicity of $M^{\lambda}$ in $\mathrm{Sym}^2(M^{(n-k,1^k)})$ is independent of $n$, as long as $n$ is sufficiently large (more precisely, if $n\geq 2k+1$).

Our result is motivated by the relative absence of explicit results on the symmetric and exterior Kronecker coefficients, i.e. the multiplicities of irreducibles in the symmetric and exterior squares of irreducible $\mathfrak{S}_n$-representations, compared to the well-investigated topic of Kronecker coefficients. The latter is analysed using various tools, such as symmetric functions (see \cite{Rosas2}), colored Yamanouchi tableaux (see \cite{Blasiak}), or invariant theory of $\textrm{GL}(V)$-representations (see \cite{SamSnowden}).
In contrast, the symmetric and exterior squares are considerably less studied. In \cite{BCI} the case of $\mu$ being a rectangle is examined, moreover, in \cite{Wolosz} the second named author characterized whether an irreducible representation is a quotient of its own exterior square, assuming its Young diagram is of height two, of width two, or a hook diagram.

In this article we use a combinatorial approach, by analysing the action of the Young symmetrizers on colored Young tableaux that correspond to a basis of $\Lambda^k(V\oplus\mathbf{1})\otimes \Lambda^l(V\oplus\mathbf{1})$ where $V\oplus\mathbf{1}$ is the $n$-dimensional permutation representation of $\mathfrak{S}_n$.

\subsection{Example}

Let us illustrate the theorem on an example. Let $n=8$ and $k=2$, i.e. $V$ is the (7-dimensional) standard representation of $\mathfrak{S}_8$. Consider the tensor square of $\Lambda^2 V$, or equivalently, of $M^{(6,1,1)}$ (which is 21-dimensional). Then one may calculate that the multiplicities of irreducible factors, that we list here indexed by their Young diagrams:


\begin{table}[h!]
\centering
\hspace*{-0.8cm}
\begin{tabular}{l||c|c|c}
         &$(\Lambda^k V)^{\otimes 2}$& 
         $\mathrm{Sym}^2(\Lambda^k V)$&
         $\mathrm{\Lambda}^2(\Lambda^k V)$\\
         \hline\hline
         $[6,2]$    & 2 & 2 & 0\\
         $[5,3]$    & 1 & 1 & 0\\
         $[5,2,1]$  & 2 & 1 & 1\\
         $[4,2,2]$  & 1 & 1 & 0\\
         $[4,2,1^2]$& 1 & 0 & 1\\
         \hline
         $[8]$      & 1 & 1 & 0\\ 
         $[7,1]$    & 1 & 1 & 0\\
         $[6,1^2]$  & 1 & 0 & 1\\
         $[5,1^3]$  & 1 & 0 & 1\\ 
         $[4,1^4]$  & 1 & 1 & 0
\end{tabular}
\caption {Multiplicities for $n=8$ and $k=2$}
\end{table}

The decomposition of $(\Lambda^kV)^{\otimes 2}$ into symmetric and antisymmetric components is not isotypic, e.g. $M^{(5,2,1)}$ appears in both components. Note however, that it may also happen (see e.g. $M^{(6,2)}$), that an irreducible factor has multiplicity two in $(\Lambda^2V)^{\otimes 2}$, but it is not a factor of the antisymmetric component.
\pagebreak

\subsection{Idea of the proof}

Let us fix $\lambda = (q,p,2^{d_2},1^{d_1})$. First let's consider the case when $d_1$ is even. Denote by $\mathbf{1}$ the trivial representation $M^{(n)}$, and let
\[
f:\big(\Lambda^{k}(V\oplus\mathbf{1})\big)^{\otimes 2} 
\twoheadrightarrow (\Lambda^{k}V)^{\otimes 2}
\]
be the $\mathfrak{S_n}$-homomorphism induced by the natural projection $V\oplus \mathbf{1}\to V$. The action of $\mathfrak{S}_n$ on $\Lambda^k(V\oplus \mathbf{1})$ and on its tensor square is more combinatorial than $\Lambda^k(V)$, so we analyse the action of the Young symmetrizers on the covering representation $\mathrm{Dom}(f)$, and push down the results via $f$.

More precisely, we will show that there is a subspace $A\subseteq \mathrm{Dom}(f)$ such that $f|_A$ is surjective and that the following skew-symmetry holds:
\begin{equation}\label{eq:skew-symmetry}
(a_1\otimes a_2) c_{\lambda} = (-1)^{m_{\lambda}} (a_2\otimes a_1) c_{\lambda}\qquad (\forall a_1\otimes a_2\in A, \forall \lambda \vdash n),
\end{equation}
where $m_{\lambda}=\frac{d_1}{2}$, and $c_\lambda$ is the Young symmetrizer corresponding to the canonical Young tableau of $\lambda$ (i.e. which is filled with $1, 2, \dots, n$ increasingly, from top to bottom, and from left to right in each row).

The existence of such a subspace $A$ implies that if $m_{\lambda}$ is odd, then for any $b\in\mathrm{Sym}^2(\Lambda^kV)$ by Eq. \ref{eq:skew-symmetry} we have $b c_{\lambda} = (-1)^{m_\lambda}b c_{\lambda}=(-1)b c_{\lambda}=0$, so $c_\lambda$ vanishes on $\mathrm{Sym}^2(\Lambda^kV)$. Similarly, if $m_{\lambda}$ is even, then for any $b\in\mathrm{\Lambda}^2(\Lambda^kV)$ we have $b c_{\lambda} = (-1)^{m_{\lambda}+1}b c_{\lambda}=(-1)b c_{\lambda}=0$. So the theorem follows for the case when $d_1$ is even.

In the case when $d_1$ is odd, we use the branching rule of $\mathfrak{S}_n$-representations, and an induction-restriction argument to derive the result from the even case and from Remmel's theorem, recalled in Theorem \ref{thm:Remmel}. If $\lambda$ is not a double-hook, but a hook, then the statement is a by-product of the lemmas proved for the even case, see Corollary \ref{cor:simple_hooks_even}.

\subsection{Explicit multiplicities}

For the sake of completeness, let us combine the theorem with the decomposition theorem for $(\Lambda^{k}V)^{\otimes2}$:
\begin{cor}
\label{cor:detailed}
Let $\lambda\vdash n$
be a Young diagram. Then the multiplicities of $M^{\lambda}$ in $\mathrm{Sym}^{2}(\Lambda^{k}V)$
(resp. $\Lambda^{2}(\Lambda^{k}V)$) are the following:
\begin{itemize}
\item \setlength\itemsep{4pt}if $\lambda=(q,p,2^{d_{2}},1^{d_{1}})$ is
a double hook for some $q\geq p\geq2$, then
\begin{itemize}
\item \setlength\itemsep{4pt}2, if $|2k+1-n|\leq q-p$ and
$d_{1}\overset{\vphantom{k}}{\equiv}0\textrm{ (resp. }2\textrm{) mod }4$,
\item 1, if $|2k+1-n|\leq q-p$ and $2\nmid d_{1}$,
\item 1, if $|2k+1-n|= q-p+1$ and $d_{1}\equiv0\textrm{ (resp. }2\textrm{) mod }4$,
\end{itemize}
\item 1, if $\lambda=(n-m,1^{m})$ is a hook where $0\leq m\leq2\min(k,n-k-1)$
and $m\equiv0$ or $1\textrm{ mod }4$ (resp. $m\equiv2$ or $3\text{ mod }4$),
\item $0$ otherwise.
\end{itemize}
\end{cor}

\subsection{Outline of the article}

In Section \ref{sec:Preliminaries} we introduce a notation on bases in the relevant tensor product representations, moreover we derive some observations on how the tensor-component flipping $T$, and the dualization $P$ act on the subrepresentations of $\big(\Lambda^k (V\oplus 1)\big)^{\otimes 2}$. We also introduce the notion of proper swaps in Lemma \ref{lem:proper_swap} to simplify calculations in the subsequent sections.
In Section \ref{sec:even_tail} (resp.
\ref{sec:odd_tail}) we prove the case of Theorem \ref{thm:main}
when $d_{1}$ is even (resp. odd), see Prop. \ref{prop:sym_even} (resp. Prop. \ref{prop:sym_odd}). 
The case of hook representations is a
by-product of the argument in Sec. \ref{sec:even_tail} (see Cor. \ref{cor:simple_hooks_even}).

\pagebreak

\section{Preliminaries\label{sec:Preliminaries}}

Denote $\mathbb{N}^{+}=\{1,2,\dots \}$, let $n\in\mathbb{N}^+$ be fixed and consider the $n$-dimensional
permutation representation of the symmetric group $\mathfrak{S}_{n}$:
\[
U\overset{\mathrm{def}}{=}V\oplus\mathbf{1}
\]
where $V=M^{(n-1, 1)}$ is the standard and $\mathbf{1}=M^{(n)}$ is the trivial representation. The standard
$\mathbb{F}$-basis elements of $\Lambda^{k}U\otimes\Lambda^{l}U$
for any $k,l\in\mathbb{N}^{+}$ are denoted as 
\begin{equation}\label{eq:def_basis}
u_{I}\otimes u_{J}=(u_{i_{1}}\wedge\dots\wedge u_{i_{k}})\otimes(u_{j_{1}}\wedge\dots\wedge u_{j_{l}})
\end{equation}
where $I=\{i_{1},\dots,i_{k}\}$ and $J=\{j_{1},\dots,j_{l}\}$ for
some $1\leq i_{1}<\dots<i_{k}\leq n$ and $1\leq j_{1}<\dots<j_{l}\leq n$.
The action of $s\in\mathfrak{S}_{n}$ is defined as 
\[
(u_{I}\otimes u_{J})s=(u_{i_{1}s}\wedge\dots\wedge u_{i_{k}s})\otimes(u_{j_{1}s}\wedge\dots\wedge u_{j_{l}s}).
\]
Consequently, the basis of $\Lambda^{k}U\otimes\Lambda^{l}U$ defined above may be indexed with $4$-colorings as follows. Define the set of all $4$-colorings as
\[
X\overset{\mathrm{def}}{=}\big\{ x:[n]\to\{0,1,2,3\}\big\}
\]
where $[n]=\{1,2,\dots,n\}$ and $n$ is assumed to be fixed, hence omitted from the notation.
Then let
\[
X_{k,l}\overset{\mathrm{def}}{=}\big\{ 
x\in X
\,\mid\,|x^{-1}(\{1,3\})|=k,\ |x^{-1}(\{2,3\})|=l\big\}.
\]
We claim that there is a bijection between the given basis of $\Lambda^k U\otimes\Lambda^l U$ and $X_{k,l}$, based on the
four subsets $I\cap J$, $I\backslash J$, $J\backslash I$ and $[n]\backslash(I\cup J)$. Indeed, for any $x\in X_{k,l}$ define
\[
w_{x}\overset{\mathrm{def}}{=}u_{x^{-1}(\{1,3\})}\otimes u_{x^{-1}(\{2,3\})}\in\Lambda^{k}U\otimes\Lambda^{l}U.
\]
Clearly, $\{w_{x}\mid x\in X_{k,l}\}$ is the standard basis of $\Lambda^{k}U\otimes\Lambda^{l}U$
as defined in Eq. \ref{eq:def_basis}.

Define the right action of $\mathfrak{S}_{n}$ on $X$ as 
\[
xs=\big(m\mapsto x(ms^{-1})\big)\qquad(m\in[n])
\]
for all $x\in X$. Note that even though there is a bijection
on the $\mathfrak{S}_{n}$-sets $\{w_{x}\ |\ x\in X_{k,l}\}$ and $X_{k,l}\subseteq X$,
$w_{x}s$ does not necessarily equal the basis element $w_{xs}$.
Instead, 
\begin{equation}
w_{x}s=\varepsilon_{x,s}w_{xs}\label{eq:signs_def}
\end{equation}
for an appropriate choice of $\varepsilon_{x,s}\in\{1,-1\}$ for any
$x\in X_{k,l}$ and $s\in\mathfrak{S}_{n}$. 

More explicitly, we may express these signs using inversion numbers
as 
\begin{equation}
\varepsilon_{x,s}=(-1)^{N_{1}(x,s)+N_{2}(x,s)}\label{eq:signs_def_inversion}
\end{equation}
where
\begin{equation}
N_{c}(x,s)\overset{\textrm{def}}{=}
\big|\big\{(p,q)\in[n]^{2}\ |\ p<q,\ ps>qs,\ x(ps),x(qs)\in\{c,3\}\big\}\big|\label{eq:inversion}
\end{equation}
for $c\in\{1,2\}$. Indeed, if $w_{x}=u_{I}\otimes u_{J}$ as above,
then 
\[
u_{I}s=u_{i_{1}s}\wedge\dots\wedge u_{i_{k}s}=(-1)^{N_{1}(x,s)}u_{Is}
\]
where $Is=\{is\ |\ i\in I\}$, since $N_{1}(x,s)$ is the inversion
number of the permutation required to sort the sequence $(i_{1}s,\dots, i_{k}s)$
increasingly. Similarly, $u_{J}s=(-1)^{N_{2}(x,s)}u_{Js}$, hence
Eq. \ref{eq:signs_def_inversion} holds.

\subsection{Color-switch (12)\label{subsec:Color-switch (12)}}

Denote by $t=(1,2)$ the transposition of $1$ and $2$ on the set
$\{0,1,2,3\}$. Then $x\mapsto t\circ x$ gives a bijection $X\to X$
that commutes with the $\mathfrak{S}_{n}$-action i.e. $(t\circ x)s=t\circ(xs)$
for any $s\in\mathfrak{S}_{n}$. 

We will need the following elementary properties of $\varepsilon_{x,s}$ defined
in Eq. \ref{eq:signs_def}:
\begin{lem}\label{lem:epsilon_basic_props}
Let $x\in X$
and $s\in\mathfrak{S}_{n}$. Then
\begin{enumerate}
\item $\varepsilon_{t\circ x,s}=\varepsilon_{x,s}$,
\item If $s$ is a transposition $(i,j)$ such that $x(i)=x(j)\in\{1,2\}$
then $\varepsilon_{x,s}=-1$,
\item If $s$ is a transposition $(i,j)$ such that $x(i)=x(j)\in\{0,3\}$
then $\varepsilon_{x,s}=1$.
\end{enumerate}
\end{lem}

\begin{proof}
Let $N_{1}(x,s)$ and $N_{2}(x,s)$ as in Eq. \ref{eq:inversion}.
Then $N_{2}(t\circ x,s)=N_{1}(x,s)$ and $N_{1}(t\circ x,s)=N_{2}(x,s)$
by definition, hence $\varepsilon_{t\circ x}=\varepsilon_{x,s}$ holds.

For the second statement, it is enough to prove the case of $x(i)=x(j)=1$ and $i < j$
by symmetry. Then we may note that $N_{2}(x,s)=0$ and 
\[
N_{1}(x,s)=1+2|\{m\in[n]\ |\ i<m<j,\ x(m)\in\{1,3\}\}|
\]
hence $\varepsilon_{x,s}=-1$. The proof of the last statement follows
similarly.
\end{proof}

\begin{cor}\label{cor:T iso}
The linear extension of $t$,
\[
T:\Lambda^{k}U\otimes\Lambda^{l}U\to\Lambda^{l}U\otimes\Lambda^{k}U \qquad
w_{x}\mapsto w_{t\circ x}
\]
is an $\mathbb{F}\mathfrak{S}_{n}$-module isomorphism.
\end{cor}
\begin{proof}
Indeed, as $t$ commutes with the group action, Lemma \ref{lem:epsilon_basic_props}/1 implies $T(w_x)s = w_{t\circ x}s = \varepsilon_{t\circ x,s} w_{(t\circ x)s} = \varepsilon_{x,s} w_{t\circ xs} = \varepsilon_{x,s} T(w_{xs}) = T(w_x s)$ so the claim follows.
\end{proof}

\subsection{Color-switch (03)(12)\label{subsec:Color-switch (03)(12)}}
In the previous subsection we introduced the isomorphism $T$, that can be interpreted combinatorially as switching the colors 1 and 2 for the elements of $X_{k,l}$, which parametrize the standard basis of $\Lambda^k U \otimes \Lambda^l U$. Now we define a similar isomorphism, switching color 1 with 2 and color 0 with 3, at the cost of an extra sign.

Define $p:\{0,1,2,3\}\to \{0,1,2,3\}$ as $p(c)=3-c$. Clearly, $p\circ (xs) = (p \circ x)s$ for any $x\in X_{k, l}$. By definition we may write $w_x = u_I \otimes u_J$ for some $I,J\subseteq [n]$. Then
\[w_{p\circ x} = u_{I^c} \otimes u_{J^c}\]
where $I^c = [n]\backslash I$.

Consider the linear bijection
\[
P:\Lambda^{k}U\otimes\Lambda^{l}U\to\Lambda^{n-k}U\otimes\Lambda^{n-l}U
\qquad
w_x \mapsto h(x)\, w_{p\circ x}
\]
where we set
$
h(x)=(-1)^{\sum_{\alpha=1}^k (i_{\alpha} - \alpha) + \sum_{\alpha=1}^l (j_{\alpha}-\alpha)}
$
for any $x\in X_{k, l}$.

\begin{lem}
\label{lem:P iso}
$P$ is an $\mathbb{F}\mathfrak{S}_n$-isomorphism.
\end{lem}

\begin{proof}
First we show that 
\[
\hat{P_k}:\Lambda^k U \to \Lambda^{n-k}U \otimes \Lambda^n U \qquad
u_I \mapsto u_{I^c} \otimes u_I \wedge u_{I^c}
\]
is an $\mathbb{F}\mathfrak{S}_n$-isomorphism.

Indeed, define $\mathrm{sign}_I(s)\in \{1,-1\}$ by the equation $u_I s = \mathrm{sign}_I(s) u_{Is}$. Then
\[
\hat{P_k}(u_I s) = \hat{P_k}(\mathrm{sign}_I(s) u_{Is})
= \mathrm{sign}_I(s) \,u_{I^cs} \otimes u_{Is} \wedge u_{I^cs}
\qquad \quad
\]
On the other hand,
\[
\begin{aligned}
\hat{P_k}(u_I) s &= \big( u_{I^c} \otimes u_I \wedge u_{I^c} \big) s \\
&= \big(\mathrm{sign}_{I^c}(s) u_{I^cs}\big) \otimes \big(\mathrm{sign}_I(s) u_{Is}\big) \wedge 
\big(\mathrm{sign}_{I^c}(s) u_{I^cs}\big) \\
&= \mathrm{sign}_I(s) \, u_{I^cs} \otimes  u_{Is} \wedge u_{I^cs}
\end{aligned}
\]
Therefore, $\hat{P_k}$ is indeed $\mathfrak{S}_n$-equivariant. 

Note that we may express $\hat{P_k}(u_I)$ equivalently as 
\[
\hat{P}_k(u_I) = (-1)^{\sum_{\alpha=1}^k (i_{\alpha}-\alpha)} u_{I^c} \otimes (u_1 \wedge u_2 \wedge \dots \wedge u_n)
\]
because we may sort the components of $u_I \wedge u_{I^c}$ using $\sum_{\alpha=1}^k (i_{\alpha}-\alpha)$ transpositions. 

Now consider the tensor product of $\hat{P_k}$ and $\hat{P_l}$:
\begin{align*}
\hat{P_k}\otimes \hat{P_l}\,:\,\Lambda^kU\otimes\Lambda^lU &\to \Lambda^{n-k}U\otimes \Lambda^nU\otimes\Lambda^{n-l}U\otimes \Lambda^nU
\\
u_I \otimes u_J &\mapsto
(-1)^{\sum_{\alpha=1}^k (i_{\alpha} - \alpha) + \sum_{\alpha=1}^l (j_{\alpha}-\alpha)} u_{I^c} \otimes \kappa \otimes u_{J^c} \otimes \kappa
\end{align*}
where $\kappa = u_1 \wedge u_2 \wedge \dots \wedge u_n$. However, $\Lambda^n U$ is the sign representation of $\mathfrak{S_n}$ hence its square is the identity. Therefore, $\hat{P_k}\otimes \hat{P_l}$ is the same as $P$ if we identify $\mathbb{F}$ with $(\Lambda^n U)^{\otimes2}$ using $1\mapsto \kappa\otimes \kappa$, in particular, $P$ is an isomorphism.
\end{proof}

\subsection{Proper swaps\label{subsec:Proper-Swaps}}
We will need another statement about $\varepsilon_{x,s}$ in Sec.
\ref{sec:even_tail}. First let us illustrate it on an example. Let $n=9$ and
\[
(1,2,1,2,2,1,3,1,2) \in X_{5,5} \qquad s = (1,2)(4,6)(5,8) \in \mathfrak{S}_{9}
\]
i.e. $s$ is switching cells of color 1 with cells of color 2 in an order-preserving way, such that no cell of color $1$ or $2$ is missed between them. 
Note that outside of switches there may be cells of color $1$ or $2$. In other words, we may split $x$ into blocks like $(1,2\mid 1\mid 2,2,1,3,1\mid 2)$ where there are blocks where we switch all the $1$'s and $2$'s, and there are blocks not moved by the permutation.

One can check that $\varepsilon_{x,s}=1$, since each transposition contributes to the inversion numbers $N_1(x,s)$ and $N_2(x,s)$ with the same amount. In the next lemma we generalize this example.

For $i,j$ integers denote
\[
[[i,j]] \overset{\mathrm{def}}{=}
\begin{cases}
[i,j]\cap\mathbb{Z} & \textrm{if } i<j\\
[j,i]\cap\mathbb{Z} & \textrm{otherwise.}
\end{cases}
\]
We show the following:

\begin{lem}[Proper Swap Lemma]\label{lem:proper_swap}
Let $x\in X$ and $s\in\mathfrak{S}_{n}$ such that 
\begin{itemize}
\item $s=\prod_{\ell=1}^{m}(i_{\ell},j_{\ell})$ is a product of $m$ disjoint
transpositions for some $1\leq i_{1}<\dots<i_{m}\leq n$ and $1\leq j_{1}<\dots<j_{m}\leq n$,
\item for all $\ell\in[m]$, $x(i_{\ell})=1$ and $x(j_{\ell})=2$, and
\item for all $\ell\in[m]$,
\[
|\{\nu \in [[i_{\ell},j_{\ell}]] \,\mid\,
\nu s=\nu,\, x(\nu)=1\}|\equiv
|\{\nu \in [[i_{\ell},j_{\ell}]] \,\mid\, 
\nu s=\nu,\, x(\nu)=2\}|
\quad(\textrm{mod }2)
\]
\end{itemize}
Then $\varepsilon_{x,s}=1$.
\end{lem}
We call $s$ a \emph{proper swap with respect to $x$} if the assumptions
of Lemma \ref{lem:proper_swap} hold.
\begin{proof}
Let $p,q\in[n]^{2}$ such that $x(ps),x(qs)\in\{1,3\}$.
If they are both fixed points of $s$ then they clearly don't contribute to $N_1(x,s)$ by $ps=p<q=qs$. 
Similarly, if they are non-fixed points, then $p=i_{\ell_{1}}$ and $q=i_{\ell_{2}}$
for some $1\leq\ell_{1}<\ell_{2}\leq m$, hence $ps=j_{\ell_{1}}<j_{\ell_{2}}=qs$, and again they don't contribute.

Now let $p$ be a non-fixed point and $q$ a fixed point. Then $p=i_{\ell}$ for some $\ell$, and the pair contributes to $N_1(x,s)$ if and only if
$i_{\ell}=ps<qs=q<p=j_{\ell}$
i.e. if $q\in [i_{\ell},j_{\ell}]$. Similarly, if $p$ is a fixed point and $q$ is a non-fixed point, then they contribute to $N_1(x,s)$ if and only if $j_{\ell}=qs<ps=p<q=i_{\ell}$ for some $\ell$.
In short,
\[
N_1(x,s)=\sum_{\ell}\Big(
|\{\nu\in [[i_{\ell},j_{\ell}]] \,\mid\, \nu s=\nu,\, 
x(\nu)=1\}|+
|\{\nu\in [[i_{\ell},j_{\ell}]] \,\mid\, 
x(\nu)=3\}|\Big)
\]
The same holds for $N_2(x,s)$ if we replace $x(\nu)=1$ by $x(\nu)=2$. The claim follows by the third assumption.
\end{proof}

\subsection{Canonical Young symmetrizers\label{subsec:Canonical-Young-symmetrizers}}

Let $\lambda\vdash n$ be a Young diagram with rows of length $(\lambda_{1},\dots,\lambda_{h})$
for some  height $h\in\mathbb{N}^{+}$ and consider the subgroup of row-preserving
permutations 
\[
R_{\lambda}=\mathfrak{S}_{\lambda_{1}}\times\dots\times\mathfrak{S}_{\lambda_{h}}\subseteq\mathfrak{S}_{n}.
\]
Similarly, denote by $C_{\lambda}=R_{\overline{\lambda}}$ the subgroup
of column-preserving permutations, where $\overline{\lambda}$ denotes
the transpose of the diagram $\lambda$. 

We define the Young symmetrizer corresponding to (the canonical Young tableau of) $\lambda$ as 
\[
c_{\lambda}=
\sum_{a\in R_{\lambda}}a
\sum_{b\in C_{\lambda}}\mathrm{sign}(b)b
\in\mathbb{F}\mathfrak{S}_{n}.
\]
Given a fixed Young diagram, e.g. $\lambda=(5,3,2)$, we may visualize
an element such as $x=(0,1,3,0,3,2,0,1,0,2)\in X_{4,4}$ as a coloring
of the Young diagram using the set of colors $\{0,1,2,3\}$:
\[
\young(01303,201,02)
\]
This terminology implied by the visualization makes it easier to formulate statements
such as ``there are two 3's in the first row'' as a shorthand for
$\exists i,j\leq\lambda_{1}:x(i)=x(j)=3$.

\section{Double Hooks with Even Length Tail\label{sec:even_tail}}

In this section we prove the case of Theorem \ref{thm:main} where
$\lambda$ is a double hook $(q,p,2^{d_{2}},1^{d_{1}})$ and the length
of its ``tail`` $d_1$ is even:

\begin{prop}
\label{prop:sym_even}Let $\lambda\vdash n$
be a Young diagram of the form $\lambda=(q,p,2^{d_{2}},1^{d_{1}})$ for some
$q\geq p\geq2$. If $d_{1}\equiv2\ (\mathrm{mod}\,4)$ then the multiplicity of
$M^{\lambda}$ in $\mathrm{Sym}^{2}(\Lambda^{k}V)$ is zero. Similarly,
if $d_{1}\equiv0\ (\mathrm{mod}\,4)$, then the multiplicity of $M^{\lambda}$ in
$\Lambda^{2}(\Lambda^{k}V)$ is zero.
\end{prop}

Equivalently, we prove that $\mathrm{Sym}^{2}(\Lambda^{k}V)c_{\lambda}=0$
if $d_{1}\equiv2\ (\mathrm{mod}\,4)$, where $c_{\lambda}$ is the (canonical) Young-symmetrizer
corresponding to $\lambda$, hence $M^{\lambda}$ is not a summand
of $\mathrm{Sym}^{2}(\Lambda^{k}V)$, and similarly for the exterior
square.

The steps of the proof are the following: First, in Lemma \ref{lem:initial}
we show a skew-symmetry relation for $w_{x}c_{\lambda}$ in the case
of $n=6$, using the observations of Lemma \ref{lem:two_same}. Then
we prove Lemma \ref{lem:wx_vs_wxH} so we may induce these skew-symmetries
for larger diagrams. This induction is carried out in Prop. \ref{prop:first_row_uninteresting},
under the assumption that the first row contains no $1$'s or $2$'s (this corresponds to subspace $A$ mentioned in the introduction).
Note that the cases where $n\leq 5$ are covered in Prop. \ref{prop:first_row_uninteresting}, case III.
Finally, in Lemma \ref{lem:enough_to_show_uninteresting_first_row}
we show that it was enough to prove under the assumption on the first row, as the images of the basis elements of this form under the projection $(\Lambda^{k}U)^{\otimes 2}\twoheadrightarrow (\Lambda^{k}V)^{\otimes 2}$ are generating $(\Lambda^{k}V)^{\otimes 2}$.

\subsection{Base case}

Let us make some simple observations
on the annihilation of $c_{\lambda}$ on the basis vectors. Recall
from Subsec. \ref{subsec:Canonical-Young-symmetrizers} that for a
given $\lambda$ we may visualize $x$ as a $4$-colored Young diagram.
Figures here are for illustration purposes only.
\begin{lem}\label{lem:two_same}
Let $x\in X$ and $\lambda\vdash n$ a Young 
diagram.
\begin{enumerate}
\item If there are two $1$'s or two $2$'s in the same row,
then $w_{x}c_{\lambda}=0$.
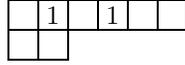
\[
\young(\ 1\ 1\ \ ,\ \ )
\]
\item If there are two $0$'s or $3$'s in the same column ,
then $w_{x}\sum_{b\in C_{\lambda}}\mathrm{sign}(b)b=0$.
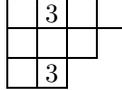
\[
\young(\ 3\ \ ,\ \ \ ,\ 3)
\]
\item If for every $a\in R_{\lambda}$ there are three (resp. five) of $0$'s
and $3$'s in the first (resp. first two) columns of $xa$ in total
then $w_{x}c_{\lambda}=0$.
\end{enumerate}
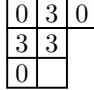
\[
\young(030,33,0\ )
\]

\end{lem}

\begin{proof}
Let $a_{0}\in R_{\lambda}$ be the transposition $(i,j)$.
The first statement follows from $c_{\lambda}=a_{0}c_{\lambda}$ (by
the definition of $c_{\lambda}$) and $w_{x}a_{0}=-w_{xa_{0}}=-w_{x}$
by Lemma \ref{lem:epsilon_basic_props}/2.

For the second statement denote $s=\sum_{b\in C_{\lambda}}\mathrm{sign}(b)b$
and let $b_{0}\in C_{\lambda}$ be the transposition $(i,j)$. Then
we have $s= -b_{0}s$ and $w_{x}b_{0}=w_{xb_{0}}=w_{x}$ by Lemma \ref{lem:epsilon_basic_props}/3.

The last statement follows from the previous one directly.
\end{proof}

\begin{lem}
\label{lem:initial}Let $\lambda\vdash6$ be a Young diagram with
$\lambda_{1}=2$. Let $x\in X_{k,l}$ be a coloring such that in the
first row of $\lambda$ there are no elements of color $1$ or $2$, and in every
other row with length at least two, there is at least one
element of color $0$ or $3$.
\begin{enumerate}
\item If $\lambda$ is $(2,2,2)$, then
\begin{equation}
w_{x}c_{\lambda}=w_{t\circ x}c_{\lambda}.\label{eq:initial222}
\end{equation}
In particular, if $k\neq l$, then both sides of Eq. \ref{eq:initial222}
are zero.
\item If $\lambda$ is $(2,2,1,1)$, $k=l$ and at least one of the elements
on the tail is of color 0 or 3, then
\begin{equation}
w_{x}c_{\lambda}=-\,w_{t\circ x}c_{\lambda}.\label{eq:initial2211}
\end{equation}
\end{enumerate}
\end{lem}

\begin{rem}
\label{rem:switch 03}
Equations \ref{eq:initial222} and \ref{eq:initial2211} can be checked one by one for the finitely many basis vectors, thus we could safely ignore  the proof of Lemma \ref{lem:initial}.  Nonetheless, we have decided to present a proof in detail, so that we can provide some explicit calculations using the notation introduced in Section \ref{sec:Preliminaries}, which might prove useful later in following the general argument.
\end{rem}

\begin{proof}
We start the proof of the lemma with four general observations on
the action $w_{x}c_{\lambda}$ to reduce the number of cases where
equations \ref{eq:initial222} and \ref{eq:initial2211} are needed
to be checked.
\\ \textbf{Preliminary observations:}

First, $w_{x}c_{\lambda}=0$ implies $w_{t\circ x}c_{\lambda}=0$,
since $T:\Lambda^{k}U\otimes\Lambda^{l}U\to\Lambda^{l}U\otimes\Lambda^{k}U$
is an $\mathbb{F}\mathfrak{S}_{n}$-module isomorphism (Cor. \ref{cor:T iso}).

Second, if $w_{x}c_{\lambda}=\rho w_{t\circ x}c_{\lambda}$ for some
$\rho\in\{1,-1\}$ and $a\in R_{\lambda}$ is an arbitrary row permutation,
then $w_{xa}c_{\lambda}=\rho w_{t\circ xa}c_{\lambda}$. Indeed 
\begin{align*}
 & w_{xa}c_{\lambda}=\varepsilon_{x,a}w_{x}ac_{\lambda}=\varepsilon_{x,a}w_{x}c_{\lambda}=\varepsilon_{x,a}\rho w_{t\circ x}c_{\lambda}=\\
 & \varepsilon_{x,a}\rho w_{t\circ x}ac_{\lambda}=\varepsilon_{x,a}\varepsilon_{t\circ x,a}\rho w_{t\circ xa}c_{\lambda}=\rho w_{t\circ xa}c_{\lambda},
\end{align*}
where equation $ac_{\lambda}=c_{\lambda}$ was used in the second and fourth
step, and Lemma \ref{lem:epsilon_basic_props}/1 in the last step.
Our second observation means that if the statement of the lemma holds
for some $w_{x}$ corresponding to a coloring $x$, then it will also
hold for any other $w_{x'}$, where $x'$ is derived from $x$ by
rearranging the colors of $x$ in the rows of $\lambda$.

Third, using the fact that $w_{t\circ t\circ x}c_{\lambda}=w_{x}c_{\lambda}$,
it can be assumed that the coloring $x$ 
contains at least as many elements of color 1 as of color 2. Moreover, in case of nonzero equality, the first element of color 1 is smaller than the
first element of color 2 (that is $\min(x^{-1}(\{1\}))<\min(x^{-1}(\{2\}))$).

Fourth, if Lemma \ref{lem:initial} holds for some basis vector $w_x$ then it also holds for $w_{p\circ t\circ x}$ i.e. when we swap the colors $0$ and $3$. Recall that $p(c)=3-c$ for any $c\in\{0,1,2,3\}$.
\[
w_x c_{\lambda} = \rho \, w_{t\circ x} c_\lambda 
\quad \Longrightarrow \quad
w_{p\circ x} c_{\lambda} = \rho \, w_{p\circ t\circ x} c_\lambda
\quad \Longrightarrow \quad
w_{p \circ t\circ x} c_{\lambda} = \rho \, w_{t \circ p\circ t\circ x} c_\lambda
\]
where $\rho \in \{1,-1\}$. Indeed, by Lemma \ref{lem:P iso} $w_x \mapsto h(x) w_{p\circ x}$ is an $\mathbb{F}\mathfrak{S}_n$-isomorphism. By definition $h(x) = h(t\circ x)$ so the first implication holds. 
For the second implication we 
used the fact that actions $p$ and $t$ are commuting on the colorings. 
\\ \textbf{Proof of Equation \ref{eq:initial222}:}

Let's determine those $w_{x}$ vectors which shall
be investigated in order the prove Eq. \ref{eq:initial222}.
By the hypotheses, it is clear that coloring $x$ contains at least
four elements of color 0 or 3. Combining Lemma \ref{lem:two_same}/3 with the first observation, we get that it will be enough to investigate those colorings $x$, which satisfy $|x^{-1}(\{0\})|=2$ and $|x^{-1}(\{3\})|=2$.
By the second observation, we may assume that the elements of color 0 or 3 are on
the following positions of $\lambda$: 
\[
\Yvcentermath1\young(**,~*,~*).
\]
We show in the next paragraph that it is sufficient to check Equation \ref{eq:initial222} for the following colorings:
\[
\Yvcentermath1x_{1}=\young(00,13,13),\ x_{2}=\young(00,13,23),\ x_{3}=\young(03,10,13),\ x_{4}=\young(03,10,23).
\]
Two cases are distinguished based on whether the colors in the
first row are equal or not. If they equal, then by the fourth observation we may assume that $x(1)=x(2)=0$
and $x(4)=x(6)=3$. If they are
not equal, then with the help of the second and fourth observations we might assume that $x(1)=0,$
$x(2)=3,$ $x(4)=0$ and $x(6)=3$.
The two remaining entries are of color 1 and 2. They might have identical
colors or different colors. Using the third observation two cases
shall be investigated $x(3)=x(5)=1$, and $x(3)=1,x(5)=2$.
Let us list the basis vectors corresponding to the identified critical colorings: 
\[
w_{x_{1}}=u_{3456}\otimes u_{46},w_{x_{2}}=u_{346}\otimes u_{456},w_{x_{3}}=u_{2356}\otimes u_{26},w_{x_{4}}=u_{236}\otimes u_{256}.
\]
We will show in details how to handle the actions $w_{x_{1}}c_{\lambda}$
and $w_{x_{4}}c_{\lambda}$, as the investigation of these actions
contain all the necessary type of computational steps needed for the
remaining two cases. We start with expanding the Young symmetrizer
$c_{\lambda}=a_{\lambda}b_{\lambda}$. Here 
$a_{\lambda}=(1+(12))(1+(34))(1+(56))$ by definition,
so 
\[
w_{x_{1}}c_{\lambda}=w_{x_{1}}a_{\lambda}b_{\lambda}=2w_{x_{1}}(1+(34))(1+(56))b_{\lambda}.
\]
Now we check the action of (1 + (34))(1 + (56)) on the coloring $x_{1}$:
\[
\Yvcentermath1x_{1}1=\young(00,13,13),\ x_{1}(34)=\young(00,31,13),\ x_{1}(56)=\young(00,13,31),\ x_{1}(34)(56)=\young(00,31,31).
\]
By Lemma \ref{lem:two_same}/2. $b_{\lambda}$ will annihilate the
terms coming from $w_{x_{1}}1$ and $w_{x_{1}}(34)(56)$, as the result
of these actions on $x_{1}$ contain two elements of color 3 in the
same column. Using this fact, and that $(35)(46)b_{\lambda}=b_{\lambda}$,
we get: 
\begin{align*}
&2w_{x_{1}}(1+(34))(1+(56))b_{\lambda}=2w_{x_{1}}((34)+(56))b_{\lambda}=\\
&2u_{3456}\otimes u_{46}(34)b_{\lambda}+2u_{3456}\otimes u_{46}(56)(35)(46)b_{\lambda}=\\&2u_{4356}\otimes u_{36}b_{\lambda}+2u_{5643}\otimes u_{63}b_{\lambda}=
-2u_{3456}\otimes u_{36}b_{\lambda}+2u_{3456}\otimes u_{36}b_{\lambda}=0.
\end{align*}
By the first observation Equation \ref{eq:initial222} follows for $w_{x_{1}}$.

Now let us consider $w_{x_{4}}c_{\lambda}$. The annihilation property of
$b_{\lambda}$ stated in Lemma \ref{lem:two_same}/2. allows us to
consider only two summands of $a_{\lambda}$, $(56)$ and $(12)(34)$. 
By this fact and that
$(135)(264)b_{\lambda}=(153)(246)b_{\lambda}=b_{\lambda}$ we get:
\begin{align*}
&w_{x_{4}}c_{\lambda}=w_{x_{4}}(1+(12))(1+(34))(1+(56))b_{\lambda}=w_{x_{4}}((56)+(12)(34))b_{\lambda}=\\
&u_{236}\otimes u_{256}((56)+(12)(34))b_{\lambda}=u_{235}\otimes u_{265}b_{\lambda}+u_{146}\otimes u_{156}b_{\lambda}=\\
&u_{235}\otimes u_{265}(135)(264)b_{\lambda}+u_{146}\otimes u_{156}(153)(246)b_{\lambda}=\\
&u_{651}\otimes u_{641}b_{\lambda}+u_{562}\otimes u_{532}b_{\lambda}=u_{156}\otimes u_{146}b_{\lambda}-u_{256}\otimes u_{235}b_{\lambda}.
\end{align*}
On the other hand 
\[
\Yvcentermath1t\circ x_{4}=\young(03,20,13)\ \ \textrm{and}\ w_{t\circ x_{4}}=u_{256}\otimes u_{236},
\]
so 
\begin{align*}
&w_{t\circ x_{4}}c_{\lambda}=w_{t\circ x_{4}}(1+(12))(1+(34))(1+(56))b_{\lambda}=w_{t\circ x_{4}}((56)+(12)(34))b_{\lambda}=\\
&u_{265}\otimes u_{235}b_{\lambda}+u_{156}\otimes u_{146}b_{\lambda}=-u_{256}\otimes u_{235}b_{\lambda}+u_{156}\otimes u_{146}b_{\lambda},
\end{align*}
which means that Equation \ref{eq:initial222} holds for $w_{x_{4}}$. For
the remaining two basis vectors we provide the raw computations. \vspace{0.5cm}
\begin{align*}
&w_{x_{2}}c_{\lambda}  =2u_{346}\otimes u_{456}((34)+(56))((35)(46))b_{\lambda}=2(-u_{456}\otimes u_{345}-u_{356}\otimes u_{346})b_{\lambda}.\\
&w_{t\circ x_{2}}c_{\lambda}  =2u_{456}\otimes u_{346}((34)+(56))b_{\lambda}=2(-u_{356}\otimes u_{346}-u_{456}\otimes u_{345})b_{\lambda}.\\
&w_{x_{3}}c_{\lambda}  =u_{2356}\otimes u_{26}((12)(34))b_{\lambda}+u_{2356}\otimes u_{26}(56)((135)(264))b_{\lambda}=0.
\end{align*}
\newpage
\textbf{Proof of Equation \ref{eq:initial2211}:} 

For the second part of the lemma, take $\lambda=(2,2,1,1)$. The hypotheses
combined with Lemma \ref{lem:two_same}/3 and the first observation give, that it is enough to investigate those colorings $x$, which satisfy $|x^{-1}(\{0\})|=|x^{-1}(\{3\})|=2$
and $|x^{-1}(\{1\})|=|x^{-1}(\{2\})|=1$. 
Using the second observation, there are two
possible configurations for the location of elements of color 0 and
3: 
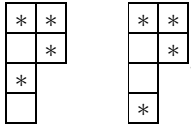
\[
\Yvcentermath1\young(**,~*,*,~) \qquad \ \young(**,~*,~,*).
\]
However, it will be sufficient to prove Eq. \ref{eq:initial2211} for
the first configuration, as we can act with $(56)$ on \ref{eq:initial2211},
providing a proof for the second type of configurations. Here we used
the fact that $(56)$ commutes with $c_{\lambda}$. Now we can list the critical colorings that are needed to be checked, just
as in the first part of the lemma: 
\[
\Yvcentermath1x_{1}=\young(00,13,3,2), \qquad x_{2}=\young(03,10,3,2).
\]
The corresponding basis vectors are the following: 
\[
w_{x_{1}}=u_{345}\otimes u_{456}, \qquad w_{x_{2}}=u_{235}\otimes u_{256}.
\]
The computations are a bit simpler than the ones we have already seen
in the first part of the lemma. 
\begin{align*}
&w_{x_{1}}c_{\lambda}=-2u_{345}\otimes u_{456}((36))b_{\lambda}=-2u_{456}\otimes u_{345}b_{\lambda}.\\
&w_{t\circ x_{1}}c_{\lambda}=2u_{456}\otimes u_{345}b_{\lambda}.\\
&w_{x_{2}}c_{\lambda}=-u_{123}\otimes u_{126}((36))b_{\lambda}=-u_{126}\otimes u_{123}b_{\lambda}.\\
&w_{t\circ x_{2}}c_{\lambda}=u_{126}\otimes u_{123}b_{\lambda}.
\end{align*}
The proof of Lemma \ref{lem:initial} is complete. 
\end{proof}

\subsection{Technical lemmas for the induction step\label{subsec:Main-lemma}}

In this subsection we prove Lemma \ref{lem:wx_vs_wxH} that is used
in the inductive step of the proof of Prop. \ref{prop:sym_even}.
First, let us consider the following simplified version of the lemma. 

Recall the definition of proper swap from Subsec. \ref{subsec:Proper-Swaps}.
\begin{lem}[Simplified Induction Lemma]
\label{lem:simplified_lemma}Let $\lambda\vdash n$ a Young diagram,
$x\in X$, $a_{0}\in R_{\lambda}$ and $b_{0}\in C_{\lambda}$
such that 
\begin{enumerate}
\item $a_{0}b_{0}$ is a proper swap with respect to $x$, 
\item $t\circ x=xa_{0}b_{0}$, and
\item $b_{0}$ centralizes $R_{\lambda}$. 
\end{enumerate}
Then $w_{x}c_{\lambda}=\mathrm{sign}(b_{0})w_{t\circ x}c_{\lambda}$.
\end{lem}

In more colorful language, the lemma says that if we may mimic the
action of $t$ on $x$ by a proper swap $a_{0}b_{0}\in R_{\lambda}C_{\lambda}$
where $b_{0}$ only moves the tail of $\lambda$, then a skew-symmetry
relation holds. For example, if $\lambda$ and $x$ are visualized
as
\[
\young(3102,012,2,3,1)
\]
then $a_{0}=(2,4)(6,7)$, $b_{0}=(8,10)$ satisfies the assumptions
of the lemma and hence $w_{x}c_{\lambda}=-w_{t\circ x}c_{\lambda}$.
Note that this may happen only if $x\in X_{k,k}$ for some $k\in \mathbb{N}^+$.
\begin{proof}
Denote $\mathrm{sign}(b_{0})=\delta$. As $b_{0}$ centralizes $R_{\lambda}$
we obtain
\[
w_{x}c_{\lambda}=\delta w_{x}\sum_{a\in R_{\lambda}}a_{0}a\sum_{b\in C_{\lambda}}\mathrm{sign}(b)b_{0}b=\delta w_{x}a_{0}b_{0}c_{\lambda}. 
\]
We may apply Lemma \ref{lem:proper_swap} to $a_{0}b_{0}$, so $\varepsilon_{x,a_{0}b_{0}}=1$.
Hence, we may continue as
\[
=\delta w_{xa_{0}b_{0}}c_{\lambda}=\delta w_{t\circ x}c_{\lambda}
\]
as we claimed.
\end{proof}

For the generalization of Lemma \ref{lem:simplified_lemma}, let us
define restrictions of Young symmetrizers. Let $H\subseteq[n]$, $\lambda\vdash n$
a Young diagram and $x\in X$. Recall that $T_{\lambda}$ denotes
its canonical Young tableau i.e. $T_{\lambda}$ is $\lambda$ filled
with $1,2,\dots,n$ row-continuously from left to right, from top
to bottom. We will say that \emph{$H$ is compatible with $\lambda$}
if the subset of $T_{\lambda}$ determined by $H$ is left-aligned
and has non-increasing row lengths. 
For example, the cells marked with '$H$' in
\[
\young(HHH\ \ \ ,\ \ \ \ ,HH\ ) \]
form a subset compatible with $\lambda$.

For any $H\subseteq[n]$
denote by $R_{\lambda}(H)$ (resp. $C_{\lambda}(H)$) the
pointwise stabilizer of $H^c=[n]\backslash H$ in $R_{\lambda}$ (resp. $C_{\lambda}$), and similarly
let
\[
R_{\lambda}(H^{c})=\mathrm{Stab}_{R_{\lambda}}(H)\qquad C_{\lambda}(H^{c})=\mathrm{Stab}_{C_{\lambda}}(H) %
\]
be the pointwise stabilizers of $H$ in $R_{\lambda}$ and $C_{\lambda}$ respectively.

Define the $H$\emph{-restricted Young symmetizer} of $\lambda$ as
\[
c_{\lambda,H}\overset{\mathrm{def}}{=}\sum_{a\in R_{\lambda}(H)}a\sum_{b\in C_{\lambda}(H)}\mathrm{sign}(b)b\in\mathrm{Sym}(H)\subseteq\mathfrak{S}_{n}
\]
We will also need an $H$-restricted notion of the color-swap $x\mapsto t\circ x$ defined in Subsec. \ref{subsec:Color-switch (12)}. 
Denote by $t_{H}:X\to X$ the map
\[
t_{H}(x)(i)=\begin{cases}
1 & \textrm{if }i\in H\textrm{ and }x(i)=2\\
2 & \textrm{if }i\in H\textrm{ and }x(i)=1\\
x(i) & \textrm{otherwise}
\end{cases}
\]
Note that while $(t\circ x)s=t\circ(xs)$ for any $s\in\mathfrak{S}_{n}$,
the same does not hold for $t_{H}$. 
Let us illustrate the action of $t_H$ using the example
for $H$ given above and some $x\in X_{3,5}$: 
\[
\begin{array}{c}
\\
\young(012123,0000,220)
\end{array}\ \overset{t_{H}}{\longmapsto}\ \begin{array}{c}
\\
\young(021123,0000,110)
\end{array}
\]

The following lemma helps us to deduce relations of the form $w_{x}c_{\lambda}=\pm w_{t\circ x}c_{\lambda}$
given similar relations with $c_{\lambda,H}$.
\begin{lem}[Induction Lemma]
\label{lem:wx_vs_wxH}Let $\lambda\vdash n$ a Young diagram, $x\in X$
and $\rho,\delta\in\{1,-1\}$. Assume that $H\subseteq[n]$ is compatible
with $\lambda$. Moreover, assume that for all $r\in R_{\lambda}$
if $w_{xr}c_{\lambda,H}$ is non-zero, then there exists $a_{r}\in R_{\lambda}(H^{c})$
and $b_{r}\in C_{\lambda}(H^{c})$ such that 
\begin{enumerate}
\item $w_{xr}c_{\lambda,H}=\rho\,w_{t_{H}(xr)}c_{\lambda,H}$, 
\item $a_{r}b_{r}$ is a proper swap with respect to $t_{H}(xr)$,
\item $t_{H}(xr)a_{r}b_{r}=t\circ xr$,
\item $b_{r}$ centralizes $R_{\lambda}(H^{c})$ and $\mathrm{sign}(b_{r})=\delta$.
\end{enumerate}
Then $w_{x}c_{\lambda}=(\rho\delta)w_{t\circ x}c_{\lambda}$.
\end{lem}

In short, if we may supplement the restricted color-swap $t_{H}$
with an element $a_{r}b_{r}\in R_{\lambda}(H^{c})C_{\lambda}(H^{c})$
such that they together mimic the action of $t$ on $xr$, where $b_{r}$
only moves the tail of $\lambda$ outside of $H$ (and this holds
for all $r\in R_{\lambda})$, then we may lift the skew-symmetry relation
(given in $(1)$) from $H$ to $[n]$.
\begin{proof}
Let $(r_{i})_{i\in I}$ resp. $(s_{j})_{j\in J}$ be a set of representatives
for the left cosets resp. right cosets of the subgroups
\[
R_{\lambda}(H)\times R_{\lambda}(H^{c})\subseteq R_{\lambda}\qquad\textrm{resp.}\qquad C_{\lambda}(H)\times C_{\lambda}(H^{c})\subseteq C_{\lambda}.
\]
In particular, we may write
\[
\sum_{r\in R_{\lambda}}r=\sum_{i\in I}\sum_{\substack{a\in R_{\lambda}(H)\\
a'\in R_{\lambda}(H^{c})
}
}r_{i}aa'\qquad\sum_{c\in C_{\lambda}}\mathrm{sign}(c)c=\sum_{j\in J}\sum_{\substack{b\in C_{\lambda}(H)\\
b'\in C_{\lambda}(H^{c})
}
}\mathrm{sign}(bb's_{j})bb's_{j}.
\]
Let's start to compute $w_{x}c_{\lambda}$. Since $R_{\lambda}(H^{c})$
centralizes $C_{\lambda}(H)$ we have
\begin{align}
w_{x}c_{\lambda} & =w_{x}\sum_{a,a',i}r_{i}aa'\sum_{b,b',j}\mathrm{sign}(bb's_{j})bb's_{j}=\label{eq:restriction_lemma}\\
 & =\sum_{i\in I}\varepsilon_{x,r_{i}}w_{xr_{i}}\sum_{a,b}\mathrm{sign}(b)ab\sum_{a',b',j}\mathrm{sign}(b's_{j})a'b's_{j}\nonumber 
\end{align}
Denote the terms
\[
c_{\lambda,H}=\sum_{a,b}\mathrm{sign}(b)ab\quad\textrm{and}\quad s_{0}=\sum_{a',b',j}\mathrm{sign}(b's_{j})a'b's_{j}.
\]
Clearly, if we denote by $I'$ the set of $i\in I$ such that
$w_{xr_{i}}c_{\lambda,H}\neq0$, 
then Eq. \ref{eq:restriction_lemma} holds with summation index $i\in I'$
as well.

For each $i\in I'$ we may choose $a_{i}\in R_{\lambda}(H^{c})$ and
$b_{i}\in C_{\lambda}(H^{c})$ corresponding to $r=r_{i}$ as in our
assumptions. Note that for these we have
\begin{equation}
s_{0}=\delta a_{i}b_{i}s_{0}\label{eq:s_0_consumes}
\end{equation}
for any $i\in I'$, using assumption (4). Note also that
\begin{equation}
c_{\lambda,H}a_{i}b_{i}=a_{i}b_{i}c_{\lambda,H}\label{eq:c_lambda_centralizing}
\end{equation}
 as $\mathrm{Sym}(H)$ centralizes $\mathrm{Sym}(H^{c})$. Recall
that condition (2) in the statement implies $w_{t_{H}(xr)}a_{r}b_{r}=w_{t_{H}(xr)a_{r}b_{r}}$
by Lemma \ref{lem:proper_swap}. Therefore we obtain
\begin{align*}
w_{x}c_{\lambda} & \overset{\textrm{Eq. }\ref{eq:restriction_lemma}}{=}\sum_{i\in I'}\varepsilon_{x,r_{i}}w_{xr_{i}}c_{\lambda,H}s_{0}\overset{\textrm{Cond. (1)}}{=}\sum_{i\in I'}\varepsilon_{x,r_{i}}\rho\,w_{t_{H}(xr_{i})}c_{\lambda,H}s_{0}=\\
 & \overset{\textrm{Eq. }\ref{eq:s_0_consumes}}{=}\sum_{i\in I'}\varepsilon_{x,r_{i}}\rho\,w_{t_{H}(xr_{i})}c_{\lambda,H}\delta a_{i}b_{i}s_{0}\overset{\textrm{Eq. }\ref{eq:c_lambda_centralizing}}{=}(\rho\delta)\sum_{i\in I'}\varepsilon_{x,r_{i}}w_{t_{H}(xr_{i})}a_{i}b_{i}c_{\lambda,H}s_{0}\\
 & \overset{\textrm{Lemma }\ref{lem:epsilon_basic_props}/1.}{=}(\rho\delta)\sum_{i\in I'}\varepsilon_{t\circ x,r_{i}}w_{t_{H}(xr_{i})}a_{i}b_{i}c_{\lambda,H}s_{0}\\
 & \overset{\textrm{Cond. (2)}}{=}(\rho\delta)\sum_{i\in I'}\varepsilon_{t\circ x,r_{i}}w_{t_{H}(xr_{i})a_{i}b_{i}}c_{\lambda,H}s_{0}\overset{\textrm{Cond. (3)}}{=}(\rho\delta)\sum_{i\in I'}\varepsilon_{t\circ x,r_{i}}w_{t\circ xr_{i}}c_{\lambda,H}s_{0}
\end{align*}
Finally, note that if $w_{xr_{i}}c_{\lambda,H}=0$ then $w_{t\circ xr_{i}}c_{\lambda,H}=0$
holds as well by 
Cor. \ref{cor:T iso} and \ref{lem:epsilon_basic_props}/1.
Hence we
may apply Eq. \ref{eq:restriction_lemma} for $t\circ x$, that gives 
\[
w_{t\circ x}c_{\lambda}=\sum_{i\in I'}\varepsilon_{t\circ x,r_{i}}w_{t\circ xr_{i}}c_{\lambda,H}s_{0}.
\]
where $I'$ is still defined as above. The claim follows.
\end{proof}

As assumption $(1)$ of Lemma \ref{lem:wx_vs_wxH} is non-trivial to prove, we claim the following:

\begin{lem}\label{start_wx_vs_wxH}
Let $\lambda\vdash n$, $x\in X_{k,l}$ a coloring of $[n]$ and $H\subseteq [n]$ of size $m$ that is compatible with $\lambda$.
Denote by $\lambda'$ the partition corresponding to $H$, and by $E$ the unique monotonically increasing $[m]\to H$ function.
Define the coloring 
\[x'=\left(x|_H \circ E\right) :[m]\to\{0,1,2,3\}.\]
Assume that for each $h_{1},h_{2}\in H$ such that $\{x(h_{1}), x(h_{2})\} = \{1, 2\}$  we have
\begin{equation} 
|\{\nu\in H^{c}\ |\ h_{1}<\nu<h_{2},\ x(\nu)=1\}|=|\{\nu\in H^{c}\ |\ h_{1}<\nu<h_{2},\ x(\nu)=2\}|\label{eq:in_the_claim2}.
\end{equation}
Then
\begin{equation}
w_{x'}c_{\lambda'}=\rho\,w_{t\circ x'}c_{\lambda'} \Longrightarrow w_{x}c_{\lambda,H}=\rho\,w_{t_{H}(x)}c_{\lambda,H} \label{eq:wz2}.
\end{equation}
\end{lem}

\begin{proof}
Denote by $f_H$ the composition of the group homomorphisms
\[\mathfrak{S}_m\overset{\cong}{\longrightarrow} \mathrm{Sym}(H) \longrightarrow \mathfrak{S}_n\]
where the first is induced by $E$ and the second by the inclusion $H\hookrightarrow [n]$.
The map $f_H$ induces an $\mathfrak{S}_m$-module structure on $\Lambda^{k}U\otimes\Lambda^{l}U$, denote this module structure as $v * r = v f_H(r)$ for any $r\in \mathfrak{S}_m$.
This definition assures that 
\begin{equation}
\label{connect}
v * c_{\lambda'} = v c_{\lambda,H}.     
\end{equation}
for any $v\in \Lambda^k U\otimes\Lambda^l U$.

Denote by $U'$ the representation $U$ but for $\mathfrak{S}_m$ instead of $\mathfrak{S}_n$, and assume that $w_{x'}\in \Lambda^{k'}U'\otimes \Lambda^{l'}U'$. Then we may consider the map of $\mathfrak{S}_m$-modules 
\[
F:\Lambda^{k'}U'\otimes \Lambda^{l'}U' \longrightarrow \Lambda^k U \otimes \Lambda^l U
\qquad w_{x'}r \mapsto w_x * r
\]
that can be checked to be a well-defined $\mathfrak{S}_m$-homomorphism, using that $w_{x'}$ generates $\Lambda^{k'}U'\otimes \Lambda^{l'}U'$.
 
First assume that $k'\neq l'$. In this case 
$w_{t_H(x)} \notin \Lambda^k U \otimes \Lambda^l U$ so both sides of the claim are zero.

Assume that $k'=l'$. By the definition of $F$ it is clear that $F(w_{t\circ x'})$ is either $w_{t_{H}(x)}$ or $-w_{t_{H}(x)}$. 
We prove that it is always the former. (Without Condition \ref{eq:in_the_claim2} the negative case could also happen, see Example \ref{exa:negative_restricted}.) This is sufficient, as then
\[
w_x c_{\lambda,H} = w_x * c_{\lambda'} 
= F(w_{x'} c_{\lambda'})
\overset{\ref{eq:wz2}}{=} F(\rho w_{t \circ x'} c_{\lambda'})=
\]
\[
= \rho F(w_{t \circ x'}) * c_{\lambda'}
= \rho w_{t_{H}(x)} * c_{\lambda'}
= \rho w_{t_{H}(x)} c_{\lambda,H}
\]
so the statement will follow in this case.

To show that $F(w_{t \circ x'})=w_{t_H(x)}$, note that there exists a proper swap $s$ with respect to $x'$ that satisfies $t\circ x' = x's$ (simply swap the first $1$ with the first $2$, the second $1$ with the second $2$, etc.).
By Lemma \ref{lem:proper_swap}, $w_{t\circ x'} = w_{x's} = w_{x'}s$.
Applying $F$ gives $F(w_{t\circ x'}) = F(w_{x'}s) = w_x*s = w_x f_H(s)$. Notice that $f_H(s)$ is a proper swap with respect to $x$. Indeed, $f_H(s)$ fulfills the first two
conditions of proper swaps by the inheritance of $s$, and fulfills the third condition of proper swaps by Condition \ref{eq:in_the_claim2} of the present lemma. 
Applying Lemma \ref{lem:proper_swap} on $f_H(s)$ gives $w_x f_H(s) = w_{x f_H(s)} = w_{t_H (x)}$, and this is exactly our claim. 
\end{proof}

\begin{exa}\label{exa:negative_restricted}
Without Condition \ref{eq:in_the_claim2}, it may happen that $F(w_{t \circ x'})=-w_{t_H(x)}$. Define $\lambda$ and $x$ as
\[
\Yvcentermath1x=\young(10,01,2,2),
\]
and let $H =\{1, 3, 5\}$. Then, $w_x = u_1\wedge u_4 \otimes u_5\wedge u_6$, $w_{t_H(x)} = u_4\wedge u_5 \otimes u_1\wedge u_6$ and
\begin{align*}
F(w_{t \circ x'}) &= F(w_{x'}(13)) = w_x * (13) = w_x(15) = u_1\wedge u_4 \otimes u_5\wedge u_6(15) \\
                  &= u_5\wedge u_4 \otimes u_1\wedge u_6 = -u_4\wedge u_5 \otimes u_1\wedge u_6.
\end{align*} 
hence Condition \ref{eq:in_the_claim2} is indeed required.
\end{exa}

We will also need an $H$-restricted version of Lemma \ref{lem:two_same}/3 in the next subsection.
\begin{lem}
\label{lem:H-restricted_five03}Let $\lambda\vdash n$, $x\in X$
and $H\subseteq[n]$ a subset of the first two columns of $\lambda$
that is compatible with $\lambda$. If there are at least five elements
of $H$ that are of color $0$ or $3$, then $w_{x}c_{\lambda,H}=0$.
\end{lem}

\begin{proof}
Denote $b_{0}=\sum_{b\in C_{\lambda}(H)}\mathrm{sign}(b)b$. It is
enough to show that $w_{x}b_{0}=0$. Indeed, then
\[
w_{x}c_{\lambda,H}=w_{x}\sum_{a\in R_{\lambda}(H)}ab_{0}=\sum_{a\in R_{\lambda}}\varepsilon_{x,a}w_{xa}b_{0}=0
\]
as we may apply $w_{x}b_{0}=0$ for $xa$ instead of $x$ as the assumptions
of the lemma are invariant under $x\mapsto xa$ by $a\in R_{\lambda}(H)$.

To prove $w_{x}b_{0}=0$, take a transposition $s=(i,j)\in C_{\lambda}(H)$
such that $x(i)=x(j)\in\{0,3\}$ (that exists by the assumptions).
Then $xs=x$ and by Lemma \ref{lem:epsilon_basic_props}/3, 
\[
w_{x}b_{0}=w_{x}(-sb_{0})=-\varepsilon_{x,s}w_{xs}b_{0}=-w_{x}b_{0}
\]
hence $w_{x}b_{0}=0$, as we claimed.
\end{proof}

\subsection{Inductive step}
Now that all the necessary technical machinery is available, we can take a direct step toward the proof of Prop. \ref{prop:sym_even}.
In this subsection we apply Lemma \ref{lem:wx_vs_wxH} (Induction Lemma) on the double hook Young diagrams with no $1$'s or $2$'s in the first row. 

\begin{prop}
\label{prop:first_row_uninteresting}Let $\lambda\vdash n$ be a Young
diagram of the form $\lambda=(q,p,2^{d_{2}},1^{d_{1}})$ for some
$q\geq p\geq2$, $d_{1},d_{2}\geq0$. Let $x\in X_{k,k}$ for some $k$
such that $x(\{1,2,\dots,q\})\subseteq\{0,3\}$. If $d_{1}$ is even,
then 
\begin{equation}
w_{x}c_{\lambda}=(-1)^{\frac{d_{1}}{2}}w_{t\circ x}c_{\lambda}\label{eq:antisymmetry_with_d1}
\end{equation}
holds.
\end{prop}

\begin{proof}
We may assume that $w_{x}c_{\lambda}\neq0$. Indeed, otherwise $w_{t\circ x}c_{\lambda}=0$
also holds by Cor. \ref{cor:T iso}.
We may also
assume that there is at most one cell of color $1$ and at most one
cell of color $2$ in each row in $\lambda$, for the given coloring
$x$, by Lemma \ref{lem:two_same}/1. 

Let us call $m\in[n]$ \emph{unpaired} if it is the only element in
its row of color $1$ or $2$, and the row is of length at least two.
Our tactic will be to choose an $H\subseteq[n]$ that covers as many
unpaired elements as possible, and satisfies the assumptions of Lemma
\ref{lem:wx_vs_wxH} (Induction Lemma). 

There are at most two unpaired elements in total, by Lemma \ref{lem:two_same}/3
and our assumption on the first row, so we may distinguish three cases
based on the number of unpaired elements.

\vspace{0.2cm}
\textbf{Case I:} Assume that there are two unpaired elements. For
example,
\[
\young(00033,130,12,23).
\]
 Let their rows be the $m$-th and the $m'$-th row of $\lambda$.
Then we may define $H=\{1,2,m_{1},m_{2},m_{1}',m_{2}'\}$ where $m_{1}$
and $m_{2}$ denotes the first and second elements of the $m$-th
row, and similarly for $m'$. 

Let us check the assumptions of Lemma \ref{lem:wx_vs_wxH}. Let $r\in R_{\lambda}$
and assume that $w_{xr}c_{\lambda,H}\neq0$. If $H$ contains at least
five elements of color $0$ or $3$ in $xr$, then we
would have $w_{xr}c_{\lambda,H}=0$ by Lemma \ref{lem:H-restricted_five03}.
Consequently, all unpaired elements of $xr$ are contained in $H$.
Observe that for each $h_{1},h_{2}\in H\subseteq[n]$ we have
\begin{equation*}
|\{\nu\in H^{c}\ |\ h_{1}<\nu<h_{2},\ (xr)(\nu)=1\}|=|\{\nu\in H^{c}\ |\ h_{1}<\nu<h_{2},\ (xr)(\nu)=2\}|.
\end{equation*}
Indeed, all the unpaired elements are contained in $H$ so every other
row between the elements of $H$ either doesn't contain any $1$'s or $2$'s, or it does contain
one of each.
This equation and Lemma \ref{eq:initial222}/1 assures that Lemma \ref{start_wx_vs_wxH} can be applied
to the coloring $xr$ and the given choice of $H$. The lemma gives $w_{xr}c_{\lambda,H}=w_{t_{H}(xr)}c_{\lambda,H}$, hence condition $(1)$ of Lemma \ref{lem:wx_vs_wxH} is verified with $\rho = 1$.

Let us check the remaining assumptions of Lemma \ref{lem:wx_vs_wxH}.
If there are different number of $1$'s as $2$'s in $H$ then $w_{z}c_{\lambda'}$
would be zero by Lemma \ref{lem:initial}/1, hence $w_{xr}c_{\lambda,H}=0$
too, by Lemma \ref{start_wx_vs_wxH}. As we assumed this is not the case, we may define
$a_{r}\in R_{\lambda}(H^{c})$ as the product of disjoint transpositions
swapping the elements of color $1$ and $2$ that are in the same
row, and similarly $b_{r}\in C_{\lambda}(H^{c})$ as the product of
disjoint transpositions swapping the $1$'s and $2$'s on the tail
of $\lambda$. The permutation $b_{r}$ has to be chosen in a monotonic
way so $a_{r}b_{r}$ is a proper swap. Then it is straightforward
to check the assumptions $(2)$, $(3)$ and $(4)$ of Lemma \ref{lem:wx_vs_wxH}
with $\delta=\mathrm{sign}(b_{r})=(-1)^{\frac{d_{1}}{2}}$ by Lemma \ref{lem:two_same}/3. The proof of this case follows from Lemma \ref{lem:wx_vs_wxH} (Induction Lemma). 

\vspace{0.2cm}
\textbf{Case II:} Assume that there is exactly one unpaired element
and it is of color $1$ (the case of color $2$ is analoguous). For
example,
\[
\young(00033,133,1,2,2,3).
\]
Let the row of the unpaired element be the $m$-th row of $\lambda$.
Then by Lemma \ref{lem:two_same}/3 and that $d_{1}$ is even, there
is exactly one element $l$ on the tail (i.e. $l>n-d_{1}$) that is
of color $0$ or $3$. 
Moreover, as the number of $1$'s and $2$'s
in $x$ agree (i.e. $x\in X_{k,k}$ for some $k$), there is one more element of color $2$ on the tail,
than of color $1$. 

Define $j$ as the least element such that it is on the tail (i.e.
$n-d_{1}<j\leq n$), is of color $2$ and there are the same number
of elements of color $1$ as of color $2$ strictly between $n-d_{1}$
and $j$. Then we may define $H=\{1,2,m_{1},m_{2},j,l\}$ where $m_{1}$
and $m_{2}$ denotes the first and second elements of the $m$-th
row.

Let us repeat the argument of the previous case. Let $r\in R_{\lambda}$
and assume that $w_{xr}c_{\lambda,H}\neq0$. Then $H$ contains at
most four elements of color $0$ or $3$ by the same argument using
Lemma \ref{lem:H-restricted_five03}. In particular, the unpaired
element in row $m$ is contained in $H$. The definition of $j$ assures that 
for each $h_{1},h_{2}\in H\subseteq[n]$ such that $\{x(h_{1}), x(h_{2})\} = \{1, 2\}$ we have
\begin{equation*}
|\{\nu\in H^{c}\ |\ h_{1}<\nu<h_{2},\ (xr)(\nu)=1\}|=|\{\nu\in H^{c}\ |\ h_{1}<\nu<h_{2},\ (xr)(\nu)=2\}|.
\end{equation*}
Just as in the previous case, this equation and Lemma \ref{eq:initial222}/2 assures that Lemma \ref{start_wx_vs_wxH} can be applied
to coloring $xr$ and restriction $H$. It gives $w_{xr}c_{\lambda,H}=-w_{t_{H}(xr)}c_{\lambda,H}$, hence condition $(1)$ of Lemma
\ref{lem:wx_vs_wxH} (Induction Lemma) is verified with $\rho = -1$.

Define $a_{r}\in R_{\lambda}(H^{c})$ (resp. $b_{r}\in C_{\lambda}(H^{c})$)
exactly the same way as in the previous case, in particular $b_{r}$
is defined in a monotonic way. Then it is straightforward to check
the assumptions $(2)$, $(3)$ and $(4)$ of Lemma \ref{lem:wx_vs_wxH}
with $\delta=\mathrm{sign}(b_{r})=(-1)^{\frac{d_{1}-2}{2}}$, hence
$\rho\delta=(-1)^{\frac{d_{1}}{2}}$. The proof of this case follows
analogously. 
\vspace{0.2cm}

\textbf{Case III:} Assume that there are no unpaired elements, e.g.
\[
\young(00033,12,1,2).
\]
Then we may simply define $a_{0}\in R_{\lambda}$ (resp. $b_{0}\in C_{\lambda}$)
as the product of disjoint transpositions swapping the elements of
color $1$ and $2$ that are in the same row (resp. on the tail of
$\lambda$), where $b_{0}$ (in fact both) are chosen in an order-preserving
way. Then the claim follows by Lemma \ref{lem:simplified_lemma}.
\end{proof}

\pagebreak

\subsection{Proof of Proposition \ref{prop:sym_even}} 

We show that it was enough to prove for the case of no $1$'s or $2$'s
in the first row, that is done in Prop. \ref{prop:first_row_uninteresting}.

For a given $\lambda\vdash n$ Young diagram and $x\in X$ coloring,
denote by $\ell(x)$ the number of $1$'s and $2$'s in total appearing in the first row i.e.
$\ell(x)=|x^{-1}(\{1,2\})\cap[\lambda_{1}]|$,
and let 
\[
W_{\ell(x)}=\mathrm{Span}(w_{y}\ |\ y\in X_{k,k},\ \ell(y)<\ell(x))\leq(\Lambda^{k}U)^{\otimes2}. 
\]
The plan is to prove Prop. \ref{prop:sym_even} for $x$ by induction on $\ell(x)$. 

Consider the natural projection 
\[U\to U\Big/\Big(\sum_{i=1}^{n}u_{i}\Big)\cong V.\]
It induces a projection $\Lambda^{k}U\to\Lambda^{k}V$ and also a
projection $(\Lambda^{k}U)^{\otimes2}\to(\Lambda^{k}V)^{\otimes2}.$
Denote the kernel of the latter by $K$. For the inductive step, we
prove the following:
\begin{lem}
\label{lem:enough_to_show_uninteresting_first_row}Let $\lambda\vdash n$
and $x\in X_{k,k}$. If $\ell(x)\geq1$ then
\[
w_{x}c_{\lambda}\in W_{\ell(x)}c_{\lambda}+K.
\]
\end{lem}
\noindent Note that the statement is interesting only for $\ell(x)\leq 2$ by Lemma \ref{lem:two_same}/1.

\begin{proof}
Let $m\in[\lambda_{1}]$ such that $x(m)=1$ (the case of $x(m)=2$
is analogous). Write 
\[
w_{x}=\varepsilon \cdot (u_{m}\wedge u_{I})\otimes u_{J} 
\]
where $\varepsilon\in \{1,-1\}$ depending on the position of $m$ and $m\notin I\cup J\subseteq[n]$ by $x(m)=1$. Consider the element:
\begin{equation}
K\ni\bigg( \sum_{i=1}^{n}\varepsilon\cdot u_{i}\wedge u_{I}\bigg)\otimes u_{J}=w_{x}+
\sum_{i\neq m}\varepsilon \cdot(u_{i}\wedge u_{I})\otimes u_{J}.\label{eq:element_in_K}
\end{equation}
We show the following:
\begin{claim}
For given $i\neq m$ the element $((u_{i}\wedge u_{I})\otimes u_{J})c_{\lambda}$
either equals $\varepsilon w_{x}c_{\lambda}$ or is contained in $W_{\ell(x)}c_{\lambda}$.
\end{claim}

The lemma follows by the claim, because then Eq. \ref{eq:element_in_K}
gives \[zw_{x}c_{\lambda}\in W_{\ell(x)}c_{\lambda}+K\] for some positive
integer $z$.
\begin{proof}[Proof of the claim]
If $x(i)\in\{1,3\}$ then $u_{i}\wedge u_{I}=0$ so $
((u_{i}\wedge u_{I})\otimes u_{J})c_{\lambda}=0
\in W_{\ell(x)}c_{\lambda}$. 
If $x(i)=2$, then $(u_{i}\wedge u_{I})\otimes u_{J}=\pm w_{y}$ where
$y\in X_{k,k}$ is defined as 
\[
y(j)=\begin{cases}
0 & \textrm{if }j=m\\
3 & \textrm{if }j=i\\
x(j) & \textrm{otherwise}
\end{cases}
\]
Notice that $\ell(y)=\ell(x)-1$, in particular $(u_{i}\wedge u_{I})\otimes u_{J}\in W_{\ell(x)}$.

If $i>\lambda_{1}$ then $(u_{i}\wedge u_{I})\otimes u_{J}\in W_{\ell(x)}$
by definition.
If $i\leq\lambda_{1}$ and $x(i)=0$, then $(u_{i}\wedge u_{I})\otimes u_{J}\,s=(u_{m}\wedge u_{I})\otimes u_{J}$
where $s=(i,m)\in R_{\lambda}$. Consequently,
\[
\big((u_{i}\wedge u_{I})\otimes u_{J}\big)c_{\lambda}=\big((u_{i}\wedge u_{I})\otimes u_{J}\big)sc_{\lambda}=\big((u_{m}\wedge u_{I})\otimes u_{J}\big)c_{\lambda}=\varepsilon w_{x}c_{\lambda}.
\]
The claim follows.
\end{proof}
\noindent The claim proves Lemma \ref{lem:enough_to_show_uninteresting_first_row}.
\end{proof}

\begin{cor}
\label{cor:simple_hooks_even}Let $\lambda\vdash n$ be a Young diagram
of the form $\lambda=(n-m,1^{m})$ and $x\in X_{k,k}$. Then 
\[
w_{x}c_{\lambda}\equiv
(-1)^{\lfloor\frac{m}{2}\rfloor}w_{t\circ x}c_{\lambda} \quad \mathrm{mod} \ K. 
\]
\end{cor}

\begin{proof}
By Lemma \ref{lem:enough_to_show_uninteresting_first_row}, it is
enough to prove for the case when every cell in the first row of $\lambda$
is of color $0$ or $3$. Indeed, if we show that then we may prove
the statement by induction on $\ell(x)$. The case of $\ell(x)=0$
is our assumption. If we know the statement for all $y$ with
$\ell(y)<\ell(x)$ then we may express $w_{x}c_{\lambda}$ as a sum
of $w_{y}c_{\lambda}$ modulo $K$.

Assume that every cell in the first row of $\lambda$ is of color
$0$ or $3$. The number of cells of color $1$ and $2$ on the tail
of $\lambda$ are the same, by $x\in X_{k,k}$. Therefore, we may apply
Lemma \ref{lem:simplified_lemma} with $a_{0}=\mathrm{id}$ and $b_{0}$
the product of disjoint transpositions swapping the cells of color
1 with the cells of color $2$ in a monotonic way. To determine $\mathrm{sign}(b_{0})$,
note that there are at most two cells in the first column of $\lambda$
that are of color $0$ or $3$, by Lemma \ref{lem:two_same}/2, but at least
one, by the assumption on the first row. Therefore, $\mathrm{sign}(b_{0})=(-1)^{\lfloor\frac{m}{2}\rfloor}$
and the claim follows.
\end{proof}

Note that Cor. \ref{cor:simple_hooks_even}
is directly connected to Prop. \ref{prop:sym_even} through the following statement: 
\begin{lem}
\label{lem:antisymmetry}Assume that $w_{x}c_{\lambda}=(-1)^{a} w_{t\circ x}c_{\lambda}$ modulo $K$ for some $a\in \mathbb{Z}$.
If $a$ is odd then the image of $w_{x}c_{\lambda}$ in 
$\mathrm{Sym}^{2}(\Lambda^{k}V)$ is zero. Similarly, if $a$ is even then the image of $w_{x}c_{\lambda}$ in 
$\Lambda^{2}(\Lambda^{k}V)$ is zero.
\end{lem}

\begin{proof}
It follows from the fact that the image of $w_{x}$ and $w_{t\circ x}$ (resp. $-w_{t\circ x}$)
are the same in $\mathrm{Sym}^{2}(\Lambda^{k}V)$ (resp. $\Lambda^{2}(\Lambda^{k}V)$).
\end{proof}

\begin{proof}[Proof of Prop. \ref{prop:sym_even}]
By Lemma \ref{lem:antisymmetry} it is enough to show that 
\[
w_{x}c_{\lambda}\equiv
(-1)^{\frac{d_{1}}{2}}w_{t\circ x}c_{\lambda} 
\quad \mathrm{mod} \ K.
\]
If $\ell(x)=0$ then the statement is proved in Proposition \ref{prop:first_row_uninteresting}.
If $\ell(x)>0$ then the claim is proved by induction using Lemma
\ref{lem:enough_to_show_uninteresting_first_row}, analogously to
the proof of Corollary \ref{cor:simple_hooks_even}.
\end{proof}
%

\section{Double Hooks with Odd Tail\label{sec:odd_tail}}

In this section we prove the second case of Theorem \ref{thm:main} by showing the following:

\begin{prop}\label{prop:sym_odd}
Let $\lambda\vdash n$ be a Young diagram of the form
$\lambda=(q,p,2^{d_{2}},1^{d_{1}})$ for some $q\geq p\geq2$.
If $d_{1}$ is odd then the multiplicity of $M^\lambda$ in $\mathrm{Sym}^2(\Lambda^k V)$ equals the multiplicity of $M^\lambda$ in $\Lambda^2(\Lambda^k V)$.
\end{prop}

The proof is based on Frobenius reciprocity, the branching rule, the fact that we already proved the case of even length tails, and that the exact multiplicities of $(\Lambda^k V)^{\otimes 2}$ are known by Remmel's theorem.

\subsection{Branching Argument}\label{subsec:branching}
Let $\mu \vdash n-1$ be a Young diagram such that the number of rows of length one is even, i.e. it has even length tail. Denote $\mathrm{Ind} = \mathrm{Ind}^{\mathfrak{S}_n}_{\mathfrak{S}_{n-1}}$, similarly for $\mathrm{Res}$, 
and let $\langle M,N\rangle$ be the usual inner product of $\mathfrak{S}_n$-representations, i.e. $\langle M,N\rangle=\dim\mathrm{Hom}_{\mathfrak{S}_n}(M,N)$.

By Frobenius reciprocity we have
\begin{equation}\label{eq:frob}
\quad\big\langle
\mathrm{Ind}\, M^{\mu}, F(\Lambda^k V)
\big\rangle=\big\langle
M^{\mu}, \mathrm{Res}\, F(\Lambda^k V)
\big\rangle\qquad F \in \{\mathrm{Sym}^2,\Lambda^2\}.
\end{equation}
By the branching rule of $\mathfrak{S}_{n}$-representations, we may decompose the left hand side of the equation as follows. For two Young diagrams $\mu\vdash (n-1)$ and $\lambda \vdash n$ let us write $\mu \nearrow \lambda$ if and only if $\lambda$ may be obtained from $\mu$ by adding a single box to it. 
With this notation the branching rule states that
\begin{equation}\label{eq:branching_general}
\mathrm{Ind}\,M^{\mu} = \sum_{\lambda:\mu \nearrow \lambda}M^{\lambda}
\end{equation}
Denote by $\mu[i]$ the Young diagram obtained from $\mu$ by adding a box to the $i$-th column, if it exists. Let 
\begin{equation}\label{eq:lambda_mu}
\lambda = (q,p,2^{d_2},1^{d_1})\qquad\textrm{and}\qquad
\mu = (q,p,2^{d_2},1^{d_1-1})
\end{equation}
for some odd $d_1$ and $q\geq p\geq2$. Then we have $\mu[1]=\lambda$ and $\mu[2]=(q,p,2^{d_2+1},1^{d_1-2})$. 
By Eq. \ref{eq:branching_general} we get:
\begin{equation}
\begin{gathered}\label{eq:induction_decomp}
\mathrm{Ind}\,M^{\mu} = 
M^{\lambda} 
\oplus (M^{\mu[2]}\textrm{ if }d_1>1)
\oplus (M^{\mu[3]}\textrm{ if }p>2)\\
\, \oplus (M^{\mu[p+1]}\textrm{ if }q > p)
\oplus M^{\mu[q+1]}
\end{gathered}
\end{equation}
The conditional terms are defined to be zero if the condition fails.

\begin{lem}\label{lem:restriction_expanded}
Denote by $V_{n-1}$ the standard (n-2)-dimensional irreducible representation of $\mathfrak{S}_{n-1}$. Then
\[
\mathrm{Res}\, F(\Lambda^k V) \cong F(\Lambda^k V_{n-1}) \oplus F(\Lambda^{k-1} V_{n-1}) \oplus 
 \big(\Lambda^k V_{n-1} \otimes \Lambda^{k-1} V_{n-1}\big)
\]
for $F \in \{\mathrm{Sym}^2,\Lambda^2\}$.
\end{lem}

It is at least plausible that Eq. \ref{eq:frob} together with Eq. \ref{eq:induction_decomp} and Lemma \ref{lem:restriction_expanded} completely determines the multiplicities for double hooks $\lambda$ with odd length tail. We will prove this in the next subsection.

\begin{proof}[Proof of Lemma \ref{lem:restriction_expanded}]
As $\mathrm{Res}$ commutes with $F$ and $\Lambda ^k$ we have
\[
\mathrm{Res}\, F(\Lambda^k V) 
= F(\Lambda^k \mathrm{Res}\, V)
\]
Denote by $\mathbf{1}_{n-1}$ the trivial representation of $\mathfrak{S}_{n-1}$.
By definition $\mathrm{Res}\,V\cong V_{n-1} \oplus \mathbf{1}_{n-1}$. Moreover $\Lambda^k(N\oplus \mathbf{1})\cong \Lambda^k N \oplus \Lambda^{k-1} N$ for any $N$, hence:
\[
\cong F\big(\Lambda^k (V_{n-1} \oplus \mathbf{1}_{n-1})\big)
\cong F\big(\Lambda^k V_{n-1} \oplus \Lambda^{k-1} V_{n-1}\big)
\]
Finally, one can observe that $F(N_1 \oplus N_2) = F(N_1) \oplus F(N_2) \oplus (N_1 \otimes N_2)$ for any $N_1, N_2$ and $F \in \{\mathrm{Sym}^2,\Lambda^2\}$, hence the claim of the lemma follows.
\end{proof}

\begin{rem}
The argument given above is not dependent on the parity of $d_1$ i.e. with induction-restriction we may get similar equations for $d_1$ even. In the end, one could combine this argument with a simultaneous induction on 4 variables ($n$, $q$, $p$, and $d_1$, descending on $q$ and $p$) and derive some parts of Prop. \ref{prop:sym_even} too. 

This approach would have two serious drawbacks: on one hand it wouldn't solve the case of $\lambda=(q,p,2^{d_2})$, where we would need a proof similar to the one given in Sec. \ref{sec:even_tail}. Moreover, it wouldn't explain why the mod 4 value of $d_1$ appears in the answer, while we think that Lemma \ref{lem:initial} and \ref{lem:simplified_lemma} are more insightful in this regard.
\end{rem}

\subsection{Application of Remmel's theorem\label{subsec:Remmel's-theorem}}

First let us recall Remmel's theorem:

\begin{thm}[Remmel \emph{\cite{Remmel1}}, 
Rosas \emph{\cite{Rosas}}]
\label{thm:Remmel} Let $n,k,l\in\mathbb{N}^{+}$ and $\lambda\vdash n$
a Young diagram. Then the multiplicities of $M^{\lambda}$ in $\Lambda^{k}V\otimes\Lambda^{l}V$ are the following:
\begin{itemize}
\item \setlength\itemsep{6pt}if $\lambda\underset{\vphantom{f}}{=}(q,p,2^{d_{2}},1^{d_{1}})$,
$q\geq p\geq2$ is a double hook then
\begin{itemize}
\item \setlength\itemsep{6pt}$2$, if 
$|k-l|\leq d_{1}$ and $|k+l+1-n|\leq q-p$,
\item $1$, if $|k-l|\leq d_{1}$ and $|k+l+1-n|=q-p+1$,
\item $1$, if $|k-l|=d_{1}+1$ and $|k+l+1-n|\leq q-p$,
\end{itemize}
\item 1, if $\lambda=(n-m,1^{m})$ is a hook where $|k'-l'|\leq m^{k,l}\leq k'+l'$, using the notation $u'=\min(u,n-u-1)$
and 
\[
m^{k,l}=\begin{cases}
m & \textrm{if }(k=k'\textrm{ and }l=l')\textrm{ or }(k\neq k'\textrm{ and }l\neq l')\\
n-m-1 & \textrm{otherwise,}
\end{cases}
\]
\item 0 otherwise.
\end{itemize}
\end{thm}

\begin{rem} 
The notation of the statement is an alternative version of the one used in \cite[Thm. 3]{Rosas} by M. H. Rosas, where she characterized the case of multiplicity $2$ as
\[
|k-l|\leq d_1
\qquad\textrm{and}\qquad
2p-1\leq k+l-2d_2-d_1\leq 2q-1.
\]
The latter is equivalent to $|k+l+1-n|\leq q-p$ by $q+p+2d_2+d_1=n$. 

Note also that Remmel's formulation in \cite[Thm. 2.1(b)]{Remmel1} contains a mathematical typo 
on the case of $\lambda = (r,1^{n-r})$,
as he writes $c_{\lambda} = \chi(s+t-n-1 \leq r \leq s+n-t)$ instead of $c_{\lambda} = \chi(s+t-n \leq r \leq s+n-t)$, where the characteristic function $\chi$ is defined below.
\end{rem}

Let us apply the theorem for some special cases. For any statement $P$ define $\chi(P)=1$ if $P$ is true, and $0$ otherwise, in particular $\chi(a\leq b) = 1$ if and only if $a\leq b$. Moreover, denote 
\[
\psi(a,b)=
\left\{ \begin{array}{ll}
2&\textrm{if }|a| < b\\
1&\textrm{if }|a| = b\\
0&\textrm{otherwise}
\end{array} \right.
\]
Recall the definition of $\lambda$ and $\mu[i]$ from
Eq. \ref{eq:lambda_mu}.
Using the notation of the previous paragraph, by Theorem \ref{thm:Remmel}, we have 
\begin{equation}\label{eq:psi_kk}
\big\langle
M^{\lambda},
(\Lambda^kV)^{\otimes 2}
\big\rangle = 
\psi(2k+1-n,q-p+1)
\end{equation}
Moreover,
\begin{equation}\label{eq:psi_kl}
\big\langle M^{\mu},
\Lambda^k V_{n-1} \otimes \Lambda^{k-1} V_{n-1}
\big\rangle
= \left\{ 
\begin{array}{ll}
\psi(2k+1-n,q-p+1) & \textrm{if }d_1 \geq 2\\
\chi(|2k+1-n|\leq q-p) & \textrm{if }d_1 = 1 
\end{array}
\right.
\end{equation}
where $V_{n-1}$ is the standard $(n-2)$-dimensional representation of $\mathfrak{S}_{n-1}$. 

\begin{cor}\label{cor:of_Remmel}
Let $\mu = (q,p,2^{d_2},1^{d_1-1})$ for some $q\geq p\geq 2$, $d_1$ odd. 
If $q>p$ then
\begin{equation}\label{eq:cor_psi_1}
\big\langle 
M^{\mu[q+1]}\oplus M^{\mu[p+1]},
F(\Lambda^k V)
\big\rangle 
=
\big\langle 
M^{\mu},F(\Lambda^{k-1} V_{n-1}) \oplus F(\Lambda^k V_{n-1})
\big\rangle
\end{equation}
Moreover, if $q=p$ then
\begin{equation}\label{eq:cor_psi_2}
\big\langle 
M^{\mu[q+1]},
F(\Lambda^k V)
\big\rangle
=
\big\langle M^{\mu},
F(\Lambda^{k-1} V_{n-1})\oplus F(\Lambda^k V_{n-1}) \big\rangle .
\end{equation}
\end{cor}

\begin{proof}
First assume that either $F=\mathrm{Sym}^2$ and $(d_1-1)\equiv 2 \textrm{ mod } 4$ or $F=\Lambda^2$ and $(d_1-1) \equiv 0 \textrm{ mod } 4$. As $\mu$ has even tail we may apply by Prop. \ref{prop:sym_even}, and so both sides of Eq. \ref{eq:cor_psi_1} and \ref{eq:cor_psi_2} are zero.

Now assume that $F$ and $d_1$ are not as above. Then by Prop. \ref{prop:sym_even}:
\[
\big\langle M^{\mu}, F(\Lambda^k V_{n-1}) \big\rangle = 
\big\langle M^{\mu}, (\Lambda^k V_{n-1})^{\otimes 2}\big\rangle.
\]
Therefore, by Eq. \ref{eq:psi_kk} we get
\begin{align*}
\big\langle 
M^{\mu[q+1]},
F(\Lambda^k V)
\big\rangle 
=& \psi(2k+1-n,(q+1)-p+1)
\\
\big\langle 
M^{\mu},
F(\Lambda^{k-1} V_{n-1})
\big\rangle 
=& \psi(2(k-1)+1-(n-1),q-p+1)
\end{align*}
It is easy to see that if $a,b$ are integers such that $b\geq 2$ then
\[\psi(a,b) - \psi(a-1,b-1) = \psi(a+b-1,1)\] 
Hence, we get
\[
\big\langle 
M^{\mu[q+1]},
F(\Lambda^k V)
\big\rangle - 
\big\langle 
M^{\mu},
F(\Lambda^{k-1} V_{n-1})
\big\rangle =
\psi(2k-n+q-p+2,1)
\]
Similarly, we have
\begin{gather*}
\big\langle 
M^{\mu[p+1]},
F(\Lambda^k V)
\big\rangle 
-
\big\langle 
M^{\mu},
F(\Lambda^k V_{n-1})
\big\rangle 
= 
\\
= \psi(2k+1-n,q-p)
- \psi(2k+2-n,q-p+1)
\\
= - \psi(2k-n+q-p+2,1)
\end{gather*}
so the first statement follows.

For Eq. \ref{eq:cor_psi_2} an analogous computations yields
\begin{gather*}
\big\langle 
M^{\mu[q+1]},
F(\Lambda^k V)
\big\rangle
-
\big\langle M^{\mu},
F(\Lambda^{k-1} V_{n-1})\oplus F(\Lambda^k V_{n-1}) \big\rangle 
=\\
=\psi(2k+1-n,2)-\psi(2k-n,1)-\psi(2k+2-n,1)
\end{gather*}
so the claim follows from $\psi(x,2)-\psi(x-1,1)-\psi(x+1,1)=0$.
\end{proof}

Now we may prove the main proposition of the section:

\begin{proof}[Proof of Prop. \ref{prop:sym_odd}]
Let us derive recursive equations on the multiplicities. 
First assume that $d_1 > 1$, $F \in \{\mathrm{Sym}^2,\Lambda^2\}$, and consider Eq. \ref{eq:frob}:
\[
\big\langle
\mathrm{Ind}\, M^{\mu}, F(\Lambda^k V)
\big\rangle
= \big\langle
M^{\mu}, \mathrm{Res}\, F(\Lambda^k V)
\big\rangle.
\]
Expand the left hand side by Eq. \ref{eq:induction_decomp} and the right hand side by Lemma \ref{lem:restriction_expanded}, and subtract the appropriate equation in Cor. \ref{cor:of_Remmel} (depending on whether $q=p$):
\begin{gather}\label{eq:mainproof_d>1}
\big\langle
M^{\lambda}\oplus M^{\mu[2]},
F(\Lambda^k V)
\big\rangle
= 
\big\langle M^{\mu},
\Lambda^k V_{n-1} \otimes \Lambda^{k-1} V_{n-1}
\big\rangle 
=\\
\overset{\ref{eq:psi_kl}}{=}
\psi(2k+1-n,q-p+1)\notag
\end{gather}
using that $\mu[1]=\lambda$ and $\langle M^{\mu[3]},F(\Lambda^k V)\rangle = 0$ by Theorem \ref{thm:Remmel} (assuming $p>2$ so $M^{\mu[3]}$ is defined). Similarly, if $d_1=1$ then we get
\begin{equation}\label{eq:mainproof_d=1}
\big\langle
M^{\lambda},
F(\Lambda^k V)
\big\rangle
= \chi(|2k+1-n|\leq q-p)
\end{equation}
Note that the right hand sides of Eq. \ref{eq:mainproof_d>1} and \ref{eq:mainproof_d=1} 
are independent of whether $F=\mathrm{Sym}^2$ or $F=\Lambda^2$. As these equations uniquely determine each multiplicity by induction on $d_1$, the claim follows.
\end{proof}

\subsection{Proof of the main theorem}

\begin{proof}[Proof of Theorem \ref{thm:main}]
Let $\lambda\vdash n$ be a Young diagram of the form $\lambda=(q,p,2^{d_2},1^{d_1})$. 
If $d_1$ is even, then the statement follows from Prop. \ref{prop:sym_even}. 

If $d_1$ is odd, then by Theorem \ref{thm:Remmel} we know that $\langle M^{\lambda},(\Lambda^k V)^{\otimes 2}\rangle\leq 2$. 
On the other hand, by Prop. \ref{prop:sym_odd} the multiplicity of the symmetric and the exterior part are the same, so the multiplicity is either zero or one in both.

If $\lambda = (n-m,1^m)$ then by Corollary \ref{cor:simple_hooks_even} and Lemma \ref{lem:antisymmetry} we get that the multiplicity of $M^\lambda$ in $\mathrm{Sym}^2(\Lambda^k V)$ is zero if $\lfloor\frac{m}{2}\rfloor$ is odd, and similarly the multiplicity in $\Lambda^2(\Lambda^k V)$ is zero if the $\lfloor\frac{m}{2}\rfloor$ is even. 
The claim follows.
\end{proof}

Corollary \ref{cor:detailed} is directly implied by Theorem \ref{thm:main} and Remmel's Theorem \ref{thm:Remmel}.


\end{document}